\documentclass{amsart}%
\usepackage{amsfonts}
\usepackage{fancyhdr}
\usepackage{comment}
\usepackage{amssymb}

\usepackage{amsfonts, amsmath, amstext, amssymb,hyperref}

\usepackage{color,centernot}

\usepackage[dvips]{graphicx}
\usepackage{latexsym}
\usepackage{amsmath}
\usepackage{changepage}

\usepackage{graphicx}%
\usepackage[T1]{fontenc}
\newtheorem{theorem}{Theorem}
\numberwithin{theorem}{section}

\newtheorem{question}[theorem]{Question}

\newtheorem{corollary}[theorem]{Corollary}

\newtheorem{definition}[theorem]{Definition}

\newtheorem{lemma}[theorem]{Lemma}

\newtheorem{observation}[theorem]{Observation}
\newtheorem{proposition}[theorem]{Proposition}
\newtheorem{remark}[theorem]{Remark}

\newcommand{\Th}{\text{Th}}
\newcommand{\AR}{\text{AR}}

\newcommand{\Q}{\mathbb{Q}}
\newcommand{\R}{\mathbb{R}}

\newcommand{\Z}{\mathbb{Z}}
\newcommand{\N}{\mathbb{N}}

\newcommand{\A}{\mathcal{A}}
\newcommand{\B}{\mathcal{B}}
\newcommand{\CC}{\mathcal{C}}
\newcommand{\D}{\mathcal{D}}

\newcommand{\ccc}{\boldsymbol{c}}
\newcommand{\ddd}{\boldsymbol{d}}

\DeclareMathOperator*{\WW}{\bigwedge\mkern-12mu\bigwedge}
\DeclareMathOperator*{\VV}{\bigvee\mkern-12mu\bigvee}
\newcommand{\tp}{\text{tp}}

\newcommand{\lom}{^{< \omega}} 
\newcommand{\E}{\exists}

\newcommand{\iv}{^{-1}}
\newcommand{\LL}{\mathcal{L}}
\newcommand{\dom}{\text{dom}}

\newcommand{\dgsp}{\text{DgSp}}

\newcommand{\SR}{\text{SR}}

\newcommand{\supp}{\text{supp}}

\newcommand{\Fin}{\text{Fin}}
\newcommand{\pSR}{\text{pSR}}
\newcommand{\SC}{\text{SC}}
\newcommand{\K}{\mathcal{K}}
\newcommand{\cut}{\text{cut}}

\begin{document}
\title{Measuring the Complexity of Countable Presburger Models}

\author{Jason Block}
\address{
\parbox{\dimexpr\textwidth -\parindent}{ Department of Mathematics, College of William \& Mary}}
\email{jeblock@wm.edu}
\urladdr{\url{https://sites.google.com/view/jasonblockmath/}}

\maketitle

\begin{abstract}
   We take two approaches to classifying the complexity of Presburger models: Scott analysis and degree spectra. In particular, we investigate the possible Scott sentence complexities and possible degree spectra of models of Presburger arithmetic.  Many of our results will be achieved by showing how given a linear order $\mathcal{L}$, we can construct a Presburger group $P_\mathcal{L}$ that maintains much of the structure of $\mathcal{L}$. 
\end{abstract}

\section{Introduction}  
In his 1965 article \cite{Scott}, Dana Scott showed that every countable structure can be identified up to isomorphism by an infinitary sentence. That is for every countable structure $\A$, there exists an $L_{\omega_1,\omega}$  sentence $\phi$ such that given any countable $\B$, $$\B \models\phi \iff B\cong \A.$$ We call such a $\phi$ a \emph{Scott sentence} of $\A$. We define the \emph{Scott rank} of $\A$ to be the least ordinal $\alpha$ such that $\A$ has a $\Pi^{in}_{\alpha +1}$ Scott sentence. A slightly finer-grained notion then Scott rank is \emph{Scott sentence complexity}, which we define as  the least $\Gamma$ among $\{\Sigma_\alpha,\Pi_\alpha,d\text{-}\Sigma_\alpha\}_{\alpha<\omega_1}$ such that $\A$ has an infinitary $\Gamma$ Scott sentence.

While Scott rank and Scott sentence complexity measure the difficulty of describing a countable structure, the difficulty of computing a copy of such a structure is given by its \emph{degree spectrum}. Given a countable structure $\A$, we define the degree spectrum of $\A$ to be the set of Turing degrees of copies of $\A$. In this article we examine the possible Scott ranks, Scott sentence complexities and degree spectra of models of \emph{Presburger arithmetic}, which is just the first order theory of $(\Z,+,<,0,1)$. We refer to models of this theory as \emph{Presburger groups}.

Questions regarding degree spectra and Scott analysis have long been asked about \emph{Peano arithmetic}, the much studied sub-theory of $\Th(\N,+,\cdot,<,0,1)$, which includes multiplication. As shown by Tennenbaum \cite{tennenbaum1959non} in 1959, no nonstandard model of Peano arithmetic has a computable copy.  In contrast, there are many nonstandard models of Presburger arithmetic with computable copies (e.g.\ $\Q\times \Z$). Thus, although degree spectra can distinguish the standard model of Peano arithmetic, it fails to do so for Presburger arithmetic.

In \cite{scott-spec-pa}, Montalb\'an and Rossegger examined the \emph{Scott spectrum} of Peano arithmetic. The Scott spectrum of a theory $T$ is the set of all Scott ranks of countable models of $T$. Montalb\'an and Rossegger found this Scott spectrum to be $\{\alpha: \alpha=1 \text{ or }\omega\leq \alpha< \omega_1\}$. In contrast, we will see by Corollary \ref{complete scott spectrum} that Presburger arithmetic has  complete Scott spectrum. That is, its Scott spectrum contains all countable positive ordinals. Further Scott analysis was done for Peano arithmetic by Gonzalez, \L{}e\l{}yk, Rossegger and Szlufik in \cite{scott-analysis-PA} where they found the possible Scott sentence complexities for these models. Similarly, we investigate the possible Scott sentence complexities of Presburger groups in Section \ref{Scott Complexities Section}. 

Dariusz Kaloci\'nski coined the term \emph{simplest model property} to describe a theory that has a unique model of least Scott rank. It was shown in \cite{scott-spec-pa} that only the standard model of Peano arithmetic has Scott rank $1$, and so Peano arithmetic has the simplest model property. We will get the same result for Presburger arithmetic by Corollary \ref{non standard SR>1}. Thus, unlike degree spectra, Scott analysis succeeds in identifying the standard models of both Peano and Presburger arithmetic. It should be noted that in a forthcoming paper, Kaloci\'nski independently proved that Presburger arithmetic has the simplest model property using different methods.

Given a countable linear order $\LL$, we describe how to construct a countable Presburger group  $P_{\LL}$ that retains much of the structure of $\LL$. In \cite{david-dino}, Gonzalez and Rossegger investigate the possible Scott sentence complexities for linear orders. Further work on the topic was done by Gonzalez, Harrison-Trainer and Ho in \cite{tubo-david-mathew}. In particular, the authors of \cite{tubo-david-mathew} and \cite{david-dino}  were able to find all the possible Scott complexities of linear orders. We are able to carry over many of their results to Presburger groups of the form $P_\LL$. In particular, we will see in Theorem \ref{main theorem} that there exist Presburger groups with the following Scott sentence complexities: 

\begin{itemize}
    \item $\Pi_\alpha$ for all ordinals $\alpha>1$;
    \item $d$-$\Sigma_\alpha$ for all successor ordinals $\alpha>1$;
    \item $\Sigma_\alpha$ for all successor ordinals $\alpha>4$. 
    
\end{itemize}
Furthermore, we show in Theorem \ref{thm: no sigma3} that there is no Presburger group with Scott sentence complexity $\Sigma_3$. We leave open for now whether there are Presburger groups with Scott sentence complexity $\Sigma_4$. Note, other complexities not mentioned above are not possible for infinite structures in general.

To prove that no Presburger group has Scott complexity $\Sigma_3$, we must look beyond just Presburger groups of the form $P_\LL$. Furthermore, we must look beyond the class of what we call \emph{plain} Presburger groups, which we do in Section \ref{non-plain section}. The class of non-plain Presburger groups includes the recursively saturated models, which we show to all have Scott rank $2$ (Proposition \ref{recursively saturated}).

In Section \ref{dgsp chapter} we find some notable degree spectra of Presburger groups. Many (though not all) of these spectra will come from groups  of the form $P_\LL$. We will see by Theorem \ref{L'} that $\dgsp(P_\LL)=\{\ddd:\ddd'\in \dgsp(\LL)\}$. Thus, we may make use of past results about the degree spectra of linear orders to find degree spectra for Presburger groups. In particular we will see in Theorem \ref{dgsp theorem} that given any $\alpha\neq 2$, the $\alpha^{th}$ jump inversion of any Turing cone is the degree spectrum of a Presburger group. This contrasts results of Richter \cite{Richter_1981} and Knight \cite{Knight_1986} for linear orders, as they showed that no linear order can have  the $0$ or $1$ jump inversion of a nontrivial Turing cone as its degree spectrum.  

Last, we pose some open questions in Section \ref{questions}.

\subsection{Scott Analysis}

Given a language $L$, we denote the set of (finitary) first order $L$-formulas as $L_{\omega,\omega}$. We may expand to the set of infinitary first order $L$-formulas, denoted $L_{\omega_1,\omega}$, by allowing the use of the symbols $\VV$ and $\WW$, which denote countably infinite disjunctions and conjunctions respectively. Note, we still only allow quantification over finitely many variables. A formula is $\Sigma^{in}_0$ if and only if it is $\Pi^{in}_0$ if and only it is a  quantifier free $L_{\omega,\omega}$ formula. Given a countable ordinal $\alpha$, we say that a formula is $\Sigma^{in}_{\alpha}$ if it can be expressed in the form $$\VV_{i\in \N} (\E\vec y)[ \psi_i(\vec x,\vec y)]$$ where each $\psi_i$ is $\Pi^{in}_{\beta_i}$ with  $\beta_i<\alpha$. A formula is $\Pi^{in}_\alpha$ if its negation is $\Sigma^{in}_\alpha$, which is to say it can be expressed in the form $$\WW_{i\in \N} (\forall\vec y)[ \neg \psi_i(\vec x,\vec y)].$$ We say that a formula is $d$-$\Sigma^{in}_\alpha$ if it can be expressed as the conjunction of a $\Sigma^{in}_\alpha$ and a $\Pi^{in}_\alpha$ formula (and is thus the "difference" of $\Sigma^{in}_\alpha$ formulas). 

The (parameterless) Scott rank of $\A$, denoted $\SR(\A)$, is the least $\alpha>0$ that satisfies the equivalent conditions of the following theorem of Montalb\'an:

\begin{theorem}[From \cite{Scott-rank}]
\label{robuster scott rank}

Given a countable structure $\A$ and a countable ordinal $\alpha>0$, the following are equivalent:

\begin{enumerate}
    \item Every automorphism orbit in $\A$ is $\Sigma^{in}_\alpha$-definable; and
    \item $\A$ has a $\Pi^{in}_{\alpha+1}$ Scott sentence; and
    \item $\A$ is uniformly $\boldsymbol{\Delta}^0_\alpha$-categorical.\qed
\end{enumerate}
\end{theorem} 

It should be noted that a Scott rank will always be a countable ordinal. A structure $\A$ is uniformly $\boldsymbol{\Delta}^0_\alpha$-categorical if  there is a Turing operator $\Psi$ along with an oracle $S\subseteq \N$ such that for all structures $\B,\CC\cong \A $, $$\Psi^{(\D(\B)\oplus \D(\CC)\oplus S)^{(\beta)}}$$ is an isomorphism from $\B$ to $\CC$, where $\beta = \alpha$ if infinite and $\beta= \alpha-1$ if finite. That is, $1+\beta=\alpha$. Here $\D(B)$ denotes the atomic diagram of $\B$, which may be viewed as a subset of $\N$ (see e.g.\ \cite{CST1}).

We define the \emph{parameterized Scott rank} of $\A$, denoted $\pSR(\A)$, to be the least $\alpha$ such that $\SR(\A,\vec a)=\alpha$ for some tuple $\vec a\in \A\lom$. The following theorem gives a similar result to Theorem \ref{robuster scott rank}:

\begin{theorem}[From \cite{Scott-rank}]
\label{robust pSR}
Given a countable structure $\A$ and a countable ordinal $\alpha>0$, the following are equivalent:

\begin{enumerate}
    \item $\pSR(\A)\leq \alpha$; and
    \item $\A$ has a $\Sigma^{in}_{\alpha+2}$ Scott sentence; and
    \item $\A$ is $\boldsymbol{\Delta}^0_\alpha$-categorical.  \qed
\end{enumerate}
\end{theorem}

We denote the Scott sentence complexity of $\A$ as $\SC(\A)$. Recall that this is the least $\Gamma$ among $\{\Sigma_\alpha,\Pi_\alpha,d\text{-}\Sigma_\alpha\}_{\alpha<\omega_1}$ such that $\A$ has an infinitary $\Gamma$ Scott sentence. For a comprehensive overview of Scott sentence complexity, see \cite{HARRISON-TRAINOR_2022}. It was shown by A.\ Miller \cite{A-Miller} that if a structure $\A$ has a $\Sigma^{in}_\alpha$ Scott sentence and a $\Pi^{in}_\alpha$ Scott sentence with $\alpha\geq2$, then $\A$ also has a $d$-$\Sigma_\beta^{in}$ Scott sentence for some $\beta<\alpha$. This gives that there will indeed be a "least" such $\Gamma$ so long as we define $\Sigma_\alpha,\Pi_\alpha < d$-$\Sigma_\alpha$.

\subsection{Degree Spectra}
The degree spectrum of a structure $\A$ is defined as $$\dgsp(\A):= \{\ddd: \E\B\cong \A \text{ with } \deg(\B)=\ddd\}.$$ In \cite{Knight_1986}, Knight showed that the degree spectrum of a structure is either upwards closed or equal to a singleton. Furthermore, the degree spectrum of a structure $\A$ will be a singleton if and only if  there exists a finite subset  of $\A$ such that
any permutation of $\A$ that fixes this subset is an automorphism. Such structures are called \emph{automorphically trivial} (not to be confused with having a trivial automorphism group). None of the structures discussed in this paper are automorphically trivial, and so all degree spectra discussed will be upwards closed. 

Note that if $P$ is a Presburger group with a computable copy (such as $\Z$ or $\Q\times \Z$) then $\dgsp(P)$ is just the set of all Turing degrees. If there is a degree $\ccc$ such that $\dgsp(\A)=\{\ddd:\ccc\leq \ddd\}$, then we say that (the isomorphism class of) $\A$ is of degree $\ccc$. We will see by Proposition \ref{up cone} and Corollary \ref{upper cone cor} that there exists a Presburger group of each degree. However, the more interesting case is when a structure does not have any degree, and we will construct many such Presburger groups in section \ref{dgsp chapter}.

\section{Presburger Arithmetic}

\begin{definition}
    Presburger arithmetic is the first order theory of the structure $(\Z,+,<,0,1)$. 
\end{definition}

Moj\.zesz Presburger, for whom the theory is named, showed in his 1927 article \cite{Presburger} that $\Th(\Z,+,0,1)$ is decidable and axiomatized by the the axioms for torsion-free abelian groups along with the following axiom for each natural number $n\geq1$: 
$$(\forall x)(\E y)[ny=x \vee ny+1 = x \vee \cdots \vee ny + (n-1) = x].$$ Note, $ny$ is shorthand for $\overbrace{y+\cdots+y}^{\text{$n$-many}}$.
At the suggestion of Tarski, Presburger proved this result using quantifier elimination \cite{Haase}. In a supplement to \cite{Presburger}, Presburger remarked that $\Th(\Z, +, <, 0,1)$ can be shown to be decidable using a similar quantifier elimination argument.  We include ordering in our language for the remainder of this paper. 

The fact that Presburger arithmetic is complete, decidable and computably axiomatizable distinguishes it from other similar theories. The theory $\Th(\N,+,\cdot,<,0,1)$, which includes multiplication and is often called \emph{True arithmetic}, is complete but not decidable nor computably axiomatizable. Peano arithmetic, a sub-theory of true arithmetic, is computably axiomatized but is not complete by G\"odel's first incompleteness theorem \cite{Godel}. Additionally, the set of consequences of the axioms of Peano arithmetic is not decidable, as exemplified by the resolution of Hilbert's tenth problem (see e.g.\ \cite{cooper}). 

Some authors use the name Presburger arithmetic to refer to the theory of the naturals under addition instead of the integers. There is not much difference between choosing one over the other, as a model of $\Th(\Z,+,<,0,1)$ can be turned into a model of $\Th(\N,+,<,0,1)$ by simply removing the negative elements and vice versa. We choose $\Th(\Z,+,<,0,1)$ for this paper as the models we get are groups and thus more pleasant to work with. However, all of our main results still hold for models of $\Th(\N,+,<,0,1)$. 

\begin{definition}
    A model of Presburger arithmetic is called a Presburger group. 
\end{definition}
\begin{remark}[See e.g.\ \cite{llew}]
A structure $\A$ is a Presburger group if and only if $\A$ is a discretely ordered abelian group (with the constant $1$ defined as the successor of $0$) that models the following sentence for all integers $n\geq 2$:
$$(\forall x)(\E y)[ny=x \vee ny+1 = x \vee \cdots \vee ny + (n-1) = x].$$
\end{remark} 

As a result of these axioms, we get that the residue modulo $n$ of an element of a Presburger group is definable for every $n\geq 1$. This allows us to define the \emph{residue sequence} of an element of a Presburger group.

\begin{definition}
    Let $A$ be a Presburger group. We define the residue sequence of $a\in A$ as the sequence $\rho(a)=( a\mod 1, a\mod 2, a \mod 3,...)$. 
\end{definition}
 \noindent This function $\rho$ gives a homomorphism from a Presburger group to the following profinite group:

 \begin{definition}
     The (additive) group  of profinite integers $\widehat \Z$ is the inverse limit of the inverse system $\langle \Z/i\Z,f^j_i\rangle_{(I,\prec)} $ where $I$ is the positive integers, the partial order $\prec$ is defined such that $i\prec j$ if $i$ divides $j$, and $f^j_i$ is defined so that $f^j_i(x)\equiv x \mod i$. That is, $\widehat \Z$ is the collection of integer sequences $(a_i)_{i\in \Z_{>0}}$ such that $a_i\in \Z/i \Z$ and $a_i\equiv a_j \mod i$ whenever $i$ divides $j$. 
 \end{definition}
 \noindent Note, it follows from the Chinese remainder theorem that $\widehat \Z \cong \prod_{p\in P}\Z_p$ where $P$ is the set of all primes and $\Z_p$ is the $p$-adic integers. 

We have that every element of a Presburger group has residue sequence in $\widehat \Z$. Conversely, given an element $\hat r \in \widehat \Z$ we can augment any given Presburger group to include a new element with residue sequence $\hat r$. 

\begin{definition}
\label{A[r]}
    Given $\hat r\in \widehat{\Z}$ with $\hat r =(r_n)_{n\in \Z_{>0}}$ and a Presburger group $A$, we can construct a Presburger group $A[\hat r]$ by adding in an infinite element $X$ with residue sequence $\hat r$. Formally, we define  $$A[\hat r]:=\left\{a + \frac{X-r_n}{n}z/ \sim : a\in A, z\in \Z,n\in \Z_{>0}\right\}$$ where $\sim$ is an equivalence relation defined by $$a_1 + \frac{X-r_n}{n}z_1 \sim a_2 + \frac{X-r_m}{m}z_2 $$ if $mz_1=nz_2$ and $nma_1 - mr_nz_1=nma_2-nr_mz_2$. Addition on $A[\hat r]$ is defined by $$ \bigg( a_1 + \frac{X-r_n}{n}z_1 \bigg) +\bigg( a_2 + \frac{X-r_m}{m}z_2\bigg)= $$ $$ \bigg(a_1+a_2 + \frac{r_{nm}-r_n}{n}z_1 + \frac{r_{nm}-r_m}{m}z_2\bigg) + \frac{X-r_{nm}}{nm}(mz_1 + nz_2). $$
We say $$a + \frac{X-r_n}{n}z  >0$$
if either $z>0$ or if $z=0$ and $a>0$. 
\end{definition}
To see that the element $X:= 0+\frac{X-0}{1}1$ of $A[\hat r]$ has residue sequence $\hat r$, simply define $Y_n:= 0+\frac{X-r_n}{n}1$. This yields $n\cdot Y_n + r_n= X$.  Additionally, note that $X>a$ for all $a\in A$.

Let us look at some examples to understand how this equivalence relation and definition for addition work. Suppose we have $r_2=1$, $r_4=3$ and $r_8=3$. It follows from the above definition that $$\frac{X-1}{2}3\sim 3 + \frac{X-3}{4}6. $$ This can be justified intuitively as we can write $$ \frac{X-1}{2}3 = -\frac{3}{2} + \frac{3}{2}X= 3- \frac{9}{2} + \frac{3}{2}X=3 + \frac{X-3}{4}6.$$ We also have from the definition that $$\left(1+\frac{X-1}{2}5\right)+ \left(-3 +\frac{X-3}{4}3\right )=3 +\frac{X-3}{8}26.$$
Note that we can write 
\begin{align*}
    \left(1+\frac{X-1}{2}5\right)+ \left(-3 +\frac{X-3}{4}3\right )=1 - \frac{5}{2} + \frac{5}{2}X -3 -\frac{9}{4} +\frac{3}{4} X = \\
    -\frac{27}{4} + \frac{13}{4}X = 3 - \frac{78}{8} +\frac{26}{8}X = 3 +\frac{X-3}{8}26.
\end{align*}

This technique for augmenting Presburger groups, along with the natural coding power of $\widehat \Z$, allows us to easily encode upper cones into the degree spectra of Presburger groups.  We will also see an alternative construction for Presburger groups with these degree spectra in Section \ref{dgsp chapter}. 

\begin{proposition}
\label{up cone}
Given any Turing degree $\ccc$, there exists a Presburger group whose degree spectrum is $\{\ddd:\ccc\leq \ddd\}$.
\end{proposition}

\begin{proof}
It is clear that given a residue sequence $\hat r\in \widehat{\Z}$, we can compute a copy of $\Z[\hat r]$ and that given a copy of 
$\Z[\hat r]$, we can compute $\hat r$. Hence, we need only show that any set of natural numbers can be encoded into a residue sequence. 

Fix some $S\subseteq \N$. Let $\{p_n\}_{n\in \N}$ be the set of primes in increasing order. We define the element $\hat r\in \widehat{\Z}$ as follows: $$r_1=0 $$ $$r_{p_k^m}= \begin{cases} 
1 & \text{if } k\in S\\
0 & \text{if } k\notin S

\end{cases}$$ 
for all $m\geq1$. For non-prime powers $n>1$, $r_n$ is the unique integer in $\Z/n\Z$ satisfying
\begin{align*}
  r_n &\equiv r_{(p_1^{m_1})} \mod {p_1^{m_1}}\\
  &\,\,\,\vdots \\
  r_n &\equiv r_{p_i^{m_i}} \mod {p_i^{m_i}}
\end{align*}
where $p_1^{m_1}\cdots p_i^{m_i}=n $. Note, the existence and uniqueness of $r_n$ is guaranteed by the Chinese remainder theorem. This residue sequence $\hat r$ is Turing equivalent to $S$.
\end{proof}

\subsection{Divisible Ordered Abelian Groups}
\begin{definition}
    A divisible ordered abelian group is a totally ordered abelian group $V$ such that for all $x\in V$ and $q\in \Q$, we have $qx\in V$. 
\end{definition}
\noindent By $qx$ we mean the unique element of $y\in V$ that satisfies $ax=by$ where $a,b\in \Z$ and $q=\frac{a}{b}$. It should be noted that this definition is equivalent to saying that every element of $V$ is divisible by every positive $n\in \N$, and that every divisible ordered abelian group forms a vector space over $\Q$. The first order theory of divisible ordered abelian groups is a complete theory and has quantifier elimination (see e.g.\ \cite{marker2002}).

\begin{definition}
\label{arch equiv}
Let $G$ be an ordered group (not necessarily divisible). Given $x\in G$, define $|x|=\max(x,-x)$. We say that $x,y\in G$ are Archimedean equivalent (denoted $x\sim y$) if there exists a positive integer $N$ such that $N|x|>|y|$ and $N|y|>|x|$. If $N|x|<|y|$ for all $N\in \N$, then we say $x\ll y$. 
\end{definition}

It is clear that $\sim$ is an equivalence relation on $G$. Given $x\in G$, we let $[x]$ denote the equivalence class under $\sim$ containing $x$. Note, the only element of the equivalence class $[0]$ is $0$. Thus, we typically ignore this equivalence class.  

\begin{definition}
The Archimedean rank of an ordered group $G$ (denoted 
$\AR(G)$) is the set $\{[x]:0\neq x\in G\}$ ordered by $[x]<[y]$ if and only if $x\ll y$. 
\end{definition}

Given any linear order $\LL$, we can construct a divisible ordered abelian group with Archimedean rank $\LL$:

\begin{definition}
    Let $\LL$ be a linear order and let $V_l$ be a divisible ordered abelian group for each $l\in \LL$. Recall that the direct product $\prod_{l\in \LL}V_l$ is the set of functions $f: \LL \to \bigcup_{l\in \LL} V_l$ such that $f(l)\in V_l$ for all $l\in \LL$. We define the direct sum $$\bigoplus_{l\in \LL}V_l:=\left\{f\in \prod_{l\in \LL}V_l: \supp(f) \text{ is finite}\right\}$$ where $\supp(f)=\{l\in \LL:f(l)\neq 0_{V_l}\}$. Defining addition as coordinate-wise, and defining $f>0$ if $f(\max(\supp(f)))>0$, we get that this direct sum is a divisible ordered abelian group.   
\end{definition}

\begin{definition}
\label{V_L}
Given a linear order $\LL$, we define $$V_\LL:=\bigoplus_{l\in \LL} \Q.$$ 

\end{definition}

\begin{observation}
    Take any nonzero $f\in V_\LL$. Let $l=\max(\supp(f))$, and define $g_l\in V_\LL$ such that $$g_l(k)=\begin{cases}
    0 & \text{ if } k\neq l \\
    1 & \text{ if } k=l.
\end{cases}$$ We have that $f\sim g_l$, and so it is clear that $\AR(V_\LL)=\LL$.
\end{observation} 

\begin{observation}
Let $P$ be a Presburger group, and let $V$ be a divisible ordered abelian group. The structure $(V\times P,+,<,0,1)$ with $0:=(0_V,0_P), 1:=(0_V,1_P)$, addition as coordinate-wise addition and ordering as the lexicographical ordering is a Presburger group. Note,  an element $(v,p)\in V\times P$ has the same residue sequence as $p$. 
\end{observation}

\begin{definition}
\label{P_l def}
    Given a linear order $\LL$, we define the Presburger group $$P_\LL:=V_\LL\times \Z.$$ 
\end{definition}

\begin{definition}
A Presburger group $P$ is called plain if every element of $P$ has the same residue sequence as an integer. That is, given $x\in P$, there exists some $z\in \Z$ such that $\rho(x)=\rho(z)$. 
\end{definition}
\begin{proposition}
\label{plain}
If $P$ is a plain Presburger group, then there exists a divisible ordered abelian group $V$ such that $P\cong V \times \Z$. 
\end{proposition}

\begin{proof}
Consider $x\in P$. Since $P$ is plain, there exists a unique integer $z_x$ that has the same residue sequence as $x$. Define $x'=x-z_x$. Note that $x'$ has a residue sequence of all zeros. Define $V=\{x':x\in P\}$. Recall that $\rho$ is a homomorphism from $P$ to $\widehat{\Z}$. We have that $V$ is the kernel of this map, and so $V$ is a subgroup of $P$. Thus, addition and ordering are defined for $V$. Since every element of $V$ has a residue of all zeros, for all $v\in V$ and positive $n\in \N$ there is some $u\in V $ with $nu=v$. This gives that $\frac{v}{n}\in V$ for all $n$ and so we get that $qv\in V$ for all $q\in \Q, v\in V$. Hence, $V$ is a divisible ordered abelian group. 

Note that every element of $P$ is of the form $v+z$ with $v\in V$ and $z\in \Z$. If $|v|>0$ (where $|v|$ is defined as $\max(v,-v)$), then $|v|$ is larger than all integers. Hence, $v+z>0$ if and only if $v>0$ or $v=0$ and $z>0$. This gives that $P\cong V \times \Z$ as witnessed by the map $v+z \mapsto (v,z)$.  
\end{proof}

\section{Scott Complexities of Presburger Groups}\label{Scott Complexities Section}

We now examine the possible Scott sentence complexities of Presburger groups. Our findings can be summarized by the following theorem:

\begin{theorem}
\label{main theorem}
    There exist Presburger groups with the following Scott sentence complexities:

    \begin{itemize}
    \item $\Pi_\alpha$ for all countable ordinals $\alpha>1$;
    \item $d$-$\Sigma_\alpha$ for all countable successor ordinals $\alpha>1$;
    \item $\Sigma_\alpha$ for all successor ordinals $\alpha>4$. 
    
\end{itemize}
Furthermore, there is no Presburger group with Scott sentence complexity $\Sigma_3$.
\end{theorem}
This theorem follows directly from Propositions  \ref{pi_n sc}, \ref{sigma_n sc},  \ref{d-sigma_n sc}, \ref{sc alpha+2}, \ref{d sigma l + 1}, \ref{pi l + 1}, \ref{Sigma Lambda + 1} and \ref{Pi l} as well as Theorem \ref{thm: no sigma3} below. All of these Scott sentence complexities are achievable for linear orders, as shown in \cite{tubo-david-mathew} and  \cite{david-dino}. The authors of \cite{david-dino} also gave a linear order with complexity $\Sigma_4$ and showed that $\Sigma_3$ is not achievable for linear orders. We leave open for now whether there exist Presburger groups with Scott sentence complexity $\Sigma_4$ (Question \ref{Q Sigma2/3}). Other complexities not mentioned above are not possible for infinite structures in general.

As a consequence of Theorem \ref{main theorem}, we get that Presburger arithmetic has complete Scott spectrum. That is:

\begin{corollary}
\label{complete scott spectrum}
Given a countable ordinal $\alpha>0$, there exists a Presburger group with Scott rank $\alpha$. 
\end{corollary}

\begin{proof}
    This follows from the fact there exists a Presburger group with Scott sentence complexity $\Pi_{\alpha+1}$ for all $\alpha$ and from Theorem \ref{robuster scott rank}. 
\end{proof}
This distinguishes Presburger arithmetic from Peano arithmetic, as it was shown in \cite{scott-spec-pa} that no model of Peano arithmetic can have a finite Scott rank other than $1$. It was further shown in \cite{scott-analysis-PA} that $\Pi_\omega$, $\Sigma_{\omega+1}$ and $\Sigma_{\omega+2}$ are not achievable Scott sentence complexities for models of Peano arithmetic and that $\Pi_{\omega+2}$ is not achievable for true arithmetic, which further distinguishes Presburger arithmetic from these theories. 

We also get an analogous result to Corollary \ref{complete scott spectrum} for parameterized Scott rank.

\begin{corollary}
    Given a countable ordinal $\alpha>0$, there exists a Presburger group with parameterized Scott rank $\alpha$. 
\end{corollary}

\begin{proof}
    If $\SC(\A)=\Pi_{\alpha +1}$ then the least $\beta$ such that $\A$ has a $\Sigma_{\beta}^{in}$ Scott sentence is $\beta=\alpha +2$. Thus, the result follows from Theorems \ref{main theorem} and \ref{robust pSR}. 
\end{proof} 

Almost all of the Presburger groups we use to achieve these complexities are of the form $P_\LL$ where $\LL$ is an infinite linear order (see Definition \ref{P_l def}). As shown by the following lemmas, the Scott rank of such a $P_\LL$ can only differ from $\LL$ by at most $1$. 

\begin{lemma}
\label{less than 1 +}
Let $\LL$ be a countably infinite linear order. We have $\SR(P_\LL)\leq 1+\SR(\LL)$.
\end{lemma}

\begin{proof}
  By Theorem \ref{robuster scott rank}, we need only show that $P_\LL$ is uniformly $\boldsymbol{\Delta}^0_{1+\SR(\LL)}$ categorical. Take any $\A,\B\cong P_\LL$ and let $1+\alpha= \SR(\LL)$.  We can build a linear order $\Tilde{\A}\cong \LL$ as follows. First define $f:\N\to \A$ so that $f(n)=a_m$ where $m$ is the least such that $\rho(a_m)=\rho(0)$ (where $\rho(a_m)$ is the residue sequence of $a_m$), $a_m>0_\A$, and $a_m$ is not Archimedean equivalent to  $1_\A$ or any $f(l)$ with $l<n$. Define $<_f$ such that $n<_f m$ if $f(n)<f(m)$ and define $\Tilde{\A}=(\N,<_f)$. Define $g:\N\to\B$ and $\Tilde{\B}\cong \LL$ analogously. Note, both $\Tilde{\A}$ and $\Tilde{\B}$ are computable from $(\D(\A)\oplus \D(\B))'$ as it only takes one jump to decide Archimedean equivalence. By definition of Scott rank, there is some set $S$ such that we can (uniformly) compute an isomorphism $h:\Tilde{\A} \to \Tilde{\B}$ from $((\D(\A)\oplus \D(\B)\oplus S)')^{(\alpha)}= (\D(\A)\oplus \D(\B)\oplus S)^{(1+\alpha)}$. By mapping $0_\A\mapsto 0_B$, $1_\A \mapsto 1_\B$ and $f(n)\mapsto g(h(n)) $ we can generate an isomorphism from $\A$ to $\B$. This isomorphism was (uniformly) computable from $(\D(\A)\oplus \D(\B)\oplus S)^{(1+\alpha)}$, and thus $\SR(P_\LL)\leq 1+1+\alpha = 1+\SR(\LL)$. 
  \end{proof}

  \begin{lemma}
  \label{srl<srpl}
      Let $\LL$ be a countably infinite linear order. We have $\SR(\LL)\leq \SR(P_\LL)$.
  \end{lemma}
  
\begin{proof}
    We show that $\LL$ is uniformly $\boldsymbol{\Delta}^0_{\SR(P_\LL)}$-categorical. Let $\A,\B\cong \LL$ and let $1+\alpha=\SR(P_\LL)$. Using $\D(\A)$ and $\D(\B)$,  we can uniformly compute structures $\hat{\A},\hat \B\cong P_\LL$ along with injections $f:\A\to \hat\A$ and $g:\B\to \hat \B $ such that $f(x)\ll f(y)$ if and only if $x<y$ (and likewise for $g$). By definition of Scott rank, there is some $S\subseteq \N$ such that we can uniformly compute an isomorphism $h:\hat \A \to \hat \B$ from $(\D(\hat \A) \oplus \D(\hat \B)\oplus  S)^{(\alpha)} $. Thus, we get that $g\iv \circ h \circ f$ is an isomorphism from $\A\to \B$ and this isomorphism was (uniformly) computable from $(\D(\A)\oplus \D(\B)\oplus S)^\alpha$. Hence, $\SR(\LL)\leq 1+\alpha =\SR(P_\LL)$. 
\end{proof}

We also get an analogous result for parameterized Scott rank.

\begin{lemma}
\label{psr PL<1+ psr L}
If $\LL$ is a countable infinite linear order then $\pSR(P_\LL)\leq 1 + \pSR(\LL)$ and $\pSR(\LL)\leq \pSR(P_\LL)$.
\end{lemma}

\begin{proof}
    The proofs of the previous two lemmas can be adapted to show that $P_\LL$ is  $\boldsymbol{\Delta}^0_{1+\pSR(\LL)}$-categorical and that $\LL$ is  $\boldsymbol{\Delta}^0_{\pSR(P_\LL)}$-categorical.
\end{proof}

Recall that if $\alpha$ is an infinite ordinal then $1+\alpha = \alpha$. This yields the following corollary: 

\begin{corollary}
\label{psr=psr}
    If $\SR(\LL)\geq \omega$, then $\SR(P_\LL)= \SR(\LL)$. Similarly, if $\pSR(\LL)\geq \omega$, then $\pSR(P_\LL)=\pSR(\LL)$. \qed
\end{corollary}

We will also see by Theorem \ref{SCL < SCPL} that if $\SC(\LL)= \Gamma_{\alpha}$, then $\SC(P_\LL)\leq \Gamma_{1+\alpha}$ (for $\Gamma = \Pi, \Sigma$ or $d$-$\Sigma$). We conjecture that $\SC(P_\LL)=\Gamma_{1+\alpha}$ for all infinite linear orders $\LL$, though this remains open for now (Question \ref{Q SRPL vs SRL}). However, we are able to show that $\SC(P_\LL)=\Gamma_{1+\alpha}$ for enough linear orders $\LL$ to achieve all of the  Scott sentence complexities claimed in Theorem \ref{main theorem}.

\subsection{Back-and-Forth Relations}

In order to calculate Scott sentence complexities, we make use of \emph{back-and-forth relations}. For a comprehensive reference on these relations, see e.g.\ \cite{AK-book}.

\begin{definition}
    Let $\A$ and $\B$ be structures in the same language, let $\alpha$ be a countable ordinal and let $\vec a$ and $\vec b$ be tuples of the same length in $\A$ and $\B$ respectively. We say $(\A, \vec a)\leq_1(\B,\vec b)$ if every finitary $\Sigma_1$ formula true of $\vec b$ in $\B$ is also true of $\vec a$ in $\A$. Given $\alpha>1$, we say $(\A,\vec a)\leq_\alpha (\B, \vec b)$ if for all $1\leq \beta<\alpha$ and $\vec d \in \B\lom$, there exists some $\vec c \in \A\lom$ such that $(\B,\vec b,\vec d) \leq_\beta (\A,\vec a,\vec c) $. 
\end{definition}

\noindent We may also write $\A\leq_\alpha \B$ by viewing $\vec a$ and $\vec b$ as the empty tuple.  These back-and-forth relations are especially useful as a result of the following theorem of Karp:

\begin{theorem}[From \cite{Karp}]
\label{karp}
    Given an ordinal $\alpha\geq 1$, the following are equivalent: 
    \begin{enumerate}
        \item $(\A,\vec a)\leq_\alpha (\B,\vec b)$; and
        \item Every $\Pi^{in}_\alpha$ formula true of $\vec a$ in $\A$ is true of $\vec b$ in $\B$; and
        \item Every $\Sigma^{in}_\alpha$ formula true of $\vec b$ in $\B$ is true of $\vec a$ in $\A$. \qed
    \end{enumerate}
\end{theorem}

This yields us the following lemmas about the relation "$\leq_1$" on linear orders and Presburger groups:

\begin{lemma}[See e.g.\ \cite{AK-book} Chapter 15]
\label{lemma same ordering}
    Let $\A$ and $\B$ be countable linear orders with $\vec a\in \A^k$ and $\vec b\in \B^k$. We have $(\A,\vec a)\leq_1(\B,\vec b)$ if and only if for all $\vec d\in \B\lom$, there is some $\vec c\in \A\lom$ such that the ordering on $\vec b,\vec d$ in $\B$ is the same as the ordering of $\vec a,\vec c$ in $\A$. \qed
\end{lemma}

\begin{lemma}[See e.g.\ \cite{AK-book} Chapter 15] 
\label{lemma geq suborderings}
Let $\A$ and $\B$ be countable linear orders. Take $\vec a\in \A^k$ and $\vec b\in \B^k$ with $a_1<\cdots< a_k$ and $b_1<\cdots <b_k$.  Define $\A_0,...,\A_k\subseteq \A$ as the open intervals between the elements of $\vec a$. That is, $\A_0=(-\infty,a_1)$, $\A_1=(a_1,a_2)$, $\dots$, $\A_k=(a_k,\infty)$. Define $\B_0,...,\B_k$ analogously. We will have $(\A,\vec a)\leq_1 (\B,\vec b)$ if and only if for all $i\leq k$, $|\A_i|\geq |\B_i|$. \qed
    
\end{lemma}

\begin{lemma}
\label{same atomic}
    Let $P$ and $R$ be Presburger groups. If $\vec p \in P\lom$ and $\vec r\in R\lom$ satisfy the same $\Pi_0$ formulas, and if $\rho(p_i)=\rho(r_i)$ for each $i$ (recall that $\rho$ gives the residue sequence of an element), then $(R,\vec r)\leq_1 (P,\vec p)$.
\end{lemma}

\begin{proof}
 Recall that Presburger arithmetic has full quantifier elimination when we expand the language to include unaray relations $D_k$ for each integer $k>2$, where we interpret $D_k(x)$ to say that $k$ divides $x$. Since $\rho(p_i)=\rho(r_i)$ for all $i$, we still have that $\vec p$ and $\vec r$ satisfy the same $\Pi_0$ formulas, even in this expanded language. Let $\psi$ be a finitary $\Sigma_1$ formula such that $P\models \psi(\vec p)$. By quantifier elimination, there is a quantifier free formula $\psi'$ such that $\psi'(\vec x)\ \equiv \psi(\vec x)$. Thus, $P\models \psi'(\vec p)$ and so $R\models \psi'(\vec r)$ since $\vec p$ and $\vec r$ satisfy the same $\Pi_0$ formulas. Hence, $R\models \psi(\vec r)$ and so $(R,\vec r) \leq _1 (P,\vec p)$ by Karp's theorem (\ref{karp}). \end{proof}

\begin{definition}
    Given a Presburger group $P$,  $q\in \Q$ and  $p\in P$ with $\rho(p)=\rho(0)$, we define $qp$ to be the unique element $p'\in P$ that satisfies $ap=bp'$ for $a,b\in \Z$ with $\frac{a}{b}=q$. 
\end{definition}
The fact that $\rho(p)=\rho(0)$ guarantees the existence of $qp$ for all $q\in \Q$.  Note, every element of $P_\LL$ can be expressed in the form $(f,z)$ where  $z\in  \Z$ and $f:\LL\to \Q$ with $\supp(f)$ finite (see Definitions \ref{V_L} and \ref{P_l def}). 

\begin{definition}
\label{def pi}
    Given a countable linear order $\LL$, define the embedding $\pi: \LL\to P_\LL$ such that $\pi(l)=(f_l,0)$ where $f_l:\LL\to \Q$ is defined so that $f_l(l)=1$ and $f_l(m)=0$ for all $m\neq l$. 
    \end{definition}

\begin{definition}
\label{def sharp}
    Take any $p\in P$.   There is a unique tuple $(l_1,...,l_k)\in \LL\lom$ with each $l_i<l_{i+1}$, a unique tuple $(q_1,...,q_k)\in \Q^k$ with each $q_i\neq 0$ and a unique $z\in \Z$ such that $$p=q_1\pi(l_1) +\cdots + q_k\pi(l_k) + z.$$ Given a  formula $\phi(  x)$ in the language of Presburger arithmetic, define $\phi_{p}^\#(x_1,...,x_k)$ to be the same as $\phi(x)$ except with all mentions of $x$ replaced with $q_1 x_1 +\cdots + q_kx_k +z$. 
\end{definition}

\begin{definition}
\label{def tau}
    Define $\tau:P_\LL\to \LL\lom$ such that $\tau(p)=(l_1,...,l_k)$ as in the above definition. Define $t_p:\LL^{|\tau(p)|}\to P_\LL$ to be the map defined by $$t_p(x_1,...,x_k)=q_1 \pi(x_1) +\cdots + q_k\pi(x_k) + z.$$ 
\end{definition}
Note that $\phi^\#_p$ has the same complexity as $\phi$. Additionally, note that  $t_p(\tau(p))=p$ and that $\phi^\#_{\pi(l)}=\phi$ for any $l\in \LL$. These definitions can be extended in the obvious way to  handle tuples $\vec p$.

\begin{lemma}
\label{phi pound}
$P_\LL\models \phi (\vec p)$ if and only if $P_\LL\models \phi^\#_{\vec p}(\pi(\tau(\vec p))$. 
\end{lemma}

\begin{proof}
    This follows directly from the above definitions.
\end{proof}

\begin{lemma}
\label{lemma: automorphism}
    Let $\LL$ be a countable linear order with $\vec a \in \LL\lom$. Take $p\in P_\LL$ such that: \begin{itemize}
        \item $0\ll p$; and
        \item $\rho(p)=\rho(0)$; and
        \item $p$ is not a $\Q$-linear combination of $\pi(\vec a)$.
    \end{itemize} 
    There exists an automorphism $G:P_\LL\to P_\LL$ such that $G$ fixes $\pi(\vec a)$ and $G(p)=\pi(l)$ for some $l\in \LL$. 
\end{lemma}

\begin{proof}
    Since $\rho(p)=\rho(0)$, we have that $p=(v,0)$ for some $v\in V_\LL$. Let $f_a$ be as in Definition \ref{def pi} for each $a$ in $\vec a$. It is enough to show that there is an automorphism $H:V_\LL\to V_\LL$ that fixes each $f_a$ and maps $v\mapsto f_l$ for some $l\in \LL$. The existence of such an $H$ is clear since $V_\LL$ is a $\Q$-vector space with basis $\{f_l:l\in \LL\}$, and $v$ is linearly independent from $f_{a_1},...,f_{a_k}$. Taking $m\in \LL$ such that $f_m$ is archimedean equivalent to $v$,  the map $H$ that sends $v$ to $f_m$ and fixes all other $f_l$ yields the automorphism we desire.
\end{proof}

\begin{proposition}
    Take any countable linear orders $\A$ and $\B$. Define the embeddings $\pi_1:\A\to P_{\A}$ and $\pi_2:\B\to P_{\B}$ as in Definition \ref{def pi}. Given any ordinal $\alpha\geq 1$ and tuples $\vec a\in \A^k$ and $\vec b\in \B^k$, $$(\A,\vec a)\leq_\alpha (\B, \vec b) \iff (P_{\A},\pi(\vec a))\leq_{1+\alpha} (P_{\B}, \pi(\vec b)).$$ 
\end{proposition}

\begin{proof}
   We prove this by induction, beginning with the base case $\alpha=1$. Define $\tau_1:P_\A\to \A\lom$ and $\tau_2:P_\B \to \B\lom$ as in Definition \ref{def tau}. Suppose that $(\A,\vec a)\leq_1(\B,\vec b)$. We show that $(P_\A,\pi(\vec a))\leq_2 (P_\B,\pi(\vec b))$. 
Given a first move $\vec p$  in the back-and-forth game, we find a response $\vec r$ such that $(P_\A,\pi(\vec a), \vec r)\geq_1 (P_\B,\pi(\vec b),\vec p)$.      We have by Lemma \ref{lemma same ordering} that given any $\vec p\in P_\B^n $, there is some $\vec c\in \vec \A^m$ such that the ordering on $\vec b, \tau_2( p_1),...,\tau_2(p_n)$ in $\B$ is the same as the ordering on $\vec a,  c_1,...,c_m$ in $\A$ (note that $m=|\tau_2(p_1)|+\cdots +|\tau_2(p_n)|$). It follows that $(P_\A,\pi(\vec a),\pi(\vec c))$ satisfies the same $\Pi_0$ formulas as $\left(P_\B,\pi(\vec b), \pi( \tau_2(\vec p)\right)$.  Take any $\Pi_0$ formula $\phi$. We will have $P_\B\models \phi(\pi(\vec b),\vec p)$ if and only if $$P_\B\models\phi^\#_{\pi(\vec b),\vec p}\left(\pi(\vec b),\pi(\tau_2(\vec p))\right). $$ Since $(P_\A,\pi(\vec a),\pi(\vec c))$ satisfies the same $\Pi_0$ formulas as $\left(P_\B,\pi(\vec b),\pi( \tau_2(\vec p))\right)$, this holds if and only if $$P_\A\models \phi^\#_{ \pi (\vec b),\vec p}\left(\pi(\vec a),\pi( \vec c)\right).$$  Define $\vec r\in P_\A^n$ by $$\vec r= \left(t_{p_1}(c_1,...,c_{|\tau_2(p_1)|}),\dots, t_{p_n}(c_{m-|\tau_2(p_2)|+1},...,c_m)\right) $$ where $t_{p_i}$ is as in definition \ref{def tau}.   We get that $\psi^\#_{\vec p}=\psi^\#_{\vec r}$ for any $\psi$ and $\tau_1(\vec r)=\vec c$. Thus, we have $$P_\A\models \phi^\#_{ \pi (\vec a),\vec r}\left(\pi(\vec a),\pi( \vec c)\right)$$ which holds if and only if $P_\A\models \phi(\pi(\vec a),\vec r)$ since $\tau_1(\vec r)=\vec c$. This will hold for all $\Pi_0$ formulas, and we can find such an $\vec r$ given any $\vec p$. By definition of $\vec r$, we have $\rho(p_i)=\rho(r_i)$ for each $i$ as both are equal to $\rho(z_i)$ where $z_i\in \Z$ as in Definition \ref{def tau}. Hence, by Lemma \ref{same atomic}, we get that $(P_\B,\pi(\vec b),\vec p)\leq_1 (P_\A,\pi(\vec a),\vec r)$ for some $\vec r$ given any $\vec p$ and so $(P_\A,\pi(\vec a))\leq_2 (P_\B,\pi(\vec b))$. 

   For the other direction, suppose that $(P_\A,\pi(\vec a))\leq_2 (P_\B,\pi(\vec b))$. Define $\A_0,...,\A_k$ and $\B_0,...,\B_k$ as in Lemma \ref{lemma geq suborderings}. Define the formula $$\varphi_n(x_1,x_2):= \E y_1,...,y_n \WW_{i\in \N} [ix_1 <y_1 ~\&~ iy_1<y_2 ~\& ~\cdots ~\& iy_n <x_2] $$ for each $n\in \N$. Two positive elements  $p$ and $r$ of a Presburger group will satisfy $\varphi_n(r,p)$ if and only if there are at least $n$ archimedean equivalence classes between those of $r$ and $p$. Since each $\varphi_n$ is $\Sigma^{in}_2$, if $P_\B\models \varphi_n(b_i,b_{i+1})$ then $P_\A\models \varphi_n(a_i,a_{i+i})$ by Karp's theorem, and so $|\A_i|\geq |\B_i|$ for each $i$, and so $(\A,\vec a)\leq_1 (\B,\vec b)$ by Lemma \ref{lemma geq suborderings}. 

   Now for the inductive step, suppose the result holds for all $1\leq \beta <\alpha$. Suppose further that $(\A,\vec a)\leq_\alpha (\B,\vec b)$. Take any $\vec p\in P_\B^n$ and $\Pi^{in}_\gamma$ formula $\phi$ such that $1\leq \gamma<1+\alpha$. If $P_\B\models \phi(\pi(\vec b),\vec p)$, then $P_\B\models \phi^\#_{\pi(\vec b),\vec p}\left(\pi(\vec b),\pi(\tau_2(\vec p))\right)$. Since $(\A,\vec a)\leq_\alpha (\B,\vec b)$, there is some $\vec c\in \A$ such that $(\B,\vec b, \tau_2(\vec p))\leq_\gamma (\A,\vec a,\vec c)$. By inductive hypothesis, this gives that $(P_\B,\pi(\vec b), \pi(\tau_2(\vec p)))\leq_{1+\gamma}(P_\A,\pi(\vec a),\pi(\vec c))$ and so $P_\A\models \phi^\#_{\pi(\vec b),\vec p}(\pi(\vec a),\pi( \vec c))$. Defining $\vec r$ as in the first paragraph of this proof we get that $P_\A\models \phi^\#_{\pi(\vec a),\vec r}(\pi(\vec a),\pi(\vec c))$ and so $P_\A\models \phi^\#_{\pi(\vec a),\vec r}(\pi(\vec a),\pi(\tau_1(\vec r)))$ and so $P_\A \models \phi(\pi(\vec a),\vec r)$. Since this works for any $\Pi_\gamma^{in}$ formula $\phi$ with $1\leq\gamma<1+\alpha$, we have that $(P_B,\pi(\vec b),\vec p)\leq_\gamma (P_\A,\pi(\vec a),\vec r)$ for all such $\gamma$. Since we can find such an $\vec r$ given any $\vec p$, we get that $(P_\A,\pi(\vec a))\leq_{1+\alpha}(P_\B,\pi(\vec b))$. 

   For the other direction, suppose that $(P_\A, \pi(\vec a))\leq_{1+\alpha} (P_\B, \pi(\vec b))$. Given any $\vec d\in \B$ and $\beta$ with $1\leq \beta<\alpha$, there exists some $\vec r\in P_\A$ such that $(P_\B,\pi(\vec b),\pi(\vec d))\leq_{1+\beta} (P_\A, \pi(\vec a),\vec r)$. Without loss of generality, suppose each element of $\vec d$ is distinct form those in $\vec b$. Since every $\Pi^{in}_2$ formula modeled by $(P_\B,\pi(\vec b),\pi (\vec d))$ is also modeled by $(P_\A,\pi(\vec a),\vec r)$, we get that each $r$ in $\vec r$ satisfies the conditions of Lemma \ref{lemma: automorphism}.  Thus, there is an automorphism of $P_\A$ that fixes $\pi(\vec a)$ and maps $\vec r$ to $\pi(\vec c)$ for some $\vec c\in \A\lom$. Hence, we have $(P_\B,\pi(\vec b),\pi(\vec d))\leq_{1+\beta} (P_\A, \pi(\vec a),\pi(\vec c))$ and so $(\B,\vec b,\vec d)\leq_\beta (\A,\vec a,\vec c)$ by inductive hypothesis. Hence, $(\A,\vec a)\leq_\alpha (\B,\vec b) $.  \end{proof}

   \begin{corollary}
   \label{Ls iff 1+ Ps}
       Given countable linear orders $\A,\B$ and an ordinal $\alpha\geq 1$, $$\A\leq_\alpha \B \iff P_\A \leq_{1+\alpha} P_\B. $$\qed
   \end{corollary}

\subsection{Finding the Complexities}

As shown in Theorem 5.1 of \cite{Alvir-et}, no infinite structure can have a $\Sigma_2^{in}$ Scott sentence. Thus the least possible Scott sentence complexity for a Presburger group is $\Pi_2$, which is achieved by the standard model:

\begin{proposition}
    The standard Presburger group $\Z$ has Scott sentence complexity $\Pi_2$. 
\end{proposition}

\begin{proof}
    Since $\Z$ does not have a $\Sigma^{in}_2$ Scott sentence, it is enough to show that $\Z$ has a $\Pi^{in}_2$ Scott sentence. Such a sentence is given by the conjunction of the Presburger axioms (all of which are $\Pi^{in}_2$) and the sentence $$ \forall x \left[\VV_{n\in \N} x= n1~\vee ~ x+n1=0 \right]$$ where $n1$ is shorthand for $\overbrace{1+\cdots +1}^{n\text{-many}} $.     \end{proof} 
This gives that $\SR(\Z)=1$. We will see below that no other Presburger group has this Scott rank:

\begin{lemma}
\label{lemma A<2Z}
    If $\A$ is a Presburger group, then $\A\leq_2 \Z$.
\end{lemma}

\begin{proof}
    Take  any tuple $\vec b\in \Z\lom$. Since $\A$ is a Presburger group, there is a copy of $\Z$ inside of $\A$ generated by the constant $1$. Let $\vec a$ be the elements within the copy of $\Z$ in $\A$ that correspond to $\vec b$. We have that both $\vec a$ and $\vec b$ are definable, as all elements of $\Z$ are definable without quantifiers. Thus, any finitary $\Sigma_1$ formula $\phi(\vec x)$ true about $\vec a$ in $\A$ is equivalent to a finitary $\Sigma_1$ sentence $\phi'$ true in $\A$. Since $\Th(\A)=\Th(\Z)$, we have that $\Z\models \phi'$ and thus  $\Z\models\phi(\vec b)$ as well. Hence, $(\Z, \vec b)\leq_1 (\A,\vec a)$ and so $\A\leq_2 \Z$. 
\end{proof}

\begin{corollary}
\label{non standard SR>1}
    If $\A$ is a nonstandard Presburger group, then $\SR(\A)>1$. 
\end{corollary}

\begin{proof}
    Let $\A$ be a nonstandard Presburger group. Since $\A\leq_2\Z$, $\A$ cannot have a $\Pi^{in}_2$ Scott sentence and so $\SR(\A)>1$. 
\end{proof}

It should be noted that Corollary \ref{non standard SR>1} was independently proven by Dariusz Kaloci\'nski using different methods in forthcoming work. We can use this to help find a Presburger group with Scott sentence complexity $d$-$\Sigma_2$. Let $P_1$ denote the Presburger group $\Q\times \Z$. The Scott sentence complexity of $P_1$ will be evident from the fact that $\SR(P_1)=2$ but $\pSR(P_1)=1$. 

\begin{proposition}
\label{P_1 SC}
    $\SC(P_1)=d$-$\Sigma_2$.
\end{proposition}

\begin{proof}
    Since $P_1\ncong\Z$, we get that it has no $\Pi^{in}_2$ Scott sentence. By Theorem \ref{robuster scott rank}, we can show that $\SR(P_1)=2$ by demonstrating that the automorphism orbits in $P_1$ are $\Sigma^{in}_2$-definable. Note that residue sequences must be preserved under automorphism. Thus if $(q_1,z_1)$ is automorphic to $(q_2,z_2)$, then $z_1=z_2$. However, any  automorphism of $(\Q,+,<)$ can be extended to an automorphism of $P_1$. The  automorphisms of $(\Q,+,<)$ are exactly those generated by multiplying everything by a positive rational. Hence, the automorphisms of $P_1$ are exactly those defined by $(q,z)\mapsto (rq,z)$ where $r$ is some positive rational number. This gives that there are three automorphism orbits for every $z \in \Z$: $$\{(q,z):q>0\}, \, \{(q,z):q<0\}\, \text {and } \{(0,z)\}.$$ Each of these orbits can be defined by the following formulas respectively:
    $$\WW_{n\in \N} x \equiv z \mod n \,\, \&\,\, \WW_{n\in \N} \overbrace{1+\cdots + 1}^\text{$n$-many}<x,$$ 
    $$\WW_{n\in \N} x \equiv z \mod n \,\, \&\,\, \WW_{n\in \N} \overbrace{1+\cdots + 1}^\text{$n$-many}+x<0$$ and $$x=\overbrace{1+\cdots + 1}^{z\text{-many}}. $$
    Each of these formulas is $\Pi^{in}_1$, and thus $\Sigma^{in}_2$. Given a tuple of elements $\vec p$, we can define its automorphism with a conjunction of the above type sentences applied to the individual elements of $\vec p$. This conjunction will still be $\Sigma^{in}_2$, and so we get that $\SR(P_1)= 2$ and so $P_1$ has a $\Pi^{in}_3$ Scott sentence. 

    As for the parameterized Scott rank, note that  by adding a constant $c$ to the language to represent the element $(1,0)\in P_1$, we get that an isomorphism between any two copies of $P_1$ can be computed (given the atomic diagrams of the copies) simply by sending $c$ to $c$ and $1$ to $1$. Thus, $\SR(P_1,c)=1$ and so $\pSR(P_1)=1$. This gives that $P_1$ has a $\Sigma^{in}_3$ Scott sentence but no $\Sigma^{in}_2$ Scott sentence, as no infinite structure has a $\Sigma^{in}_2$ Scott sentence. Since $P_1$ has both a $\Sigma^{in}_3$ Scott sentence and a $\Pi^{in}_3$ Scott sentence, it must have a $d$-$\Sigma^{in}_2$ Scott sentence as well. 
\end{proof} 

It should be noted that if we define $P_n=\Q^n\times \Z$ for any positive integer $n$, a similar proof as above yields that $\SC(P_n)=d$-$\Sigma_2$.

Let us now show the existence of the remaining finite Scott sentence complexities claimed in Theorem \ref{main theorem}. All of the finite cases are covered by the next three propositions.

\begin{proposition}
\label{pi_n sc}
    Given a natural number $n\geq 2$, there exists a Presburger group with Scott complexity $\Pi_n$.
\end{proposition}

\begin{proof}
    For the $n=2$ case, recall that the standard Presburger group $\Z$ has Scott complexity $\Pi_2$. As noted in the proof of Corollary 3.10 of \cite{david-dino}, the linear order $\Z^n\cdot\Q$ has Scott complexity $\Pi_{2n+2}$ but $\Z^{n+1}\leq_{2n+2} \Z^n\cdot\Q$. It follows from Lemma \ref{less than 1 +} that $P_{\Z^n\cdot \Q}$ has a $\Pi^{in}_{2n+3}$ Scott sentence and from Corollary \ref{Ls iff 1+ Ps} that $P_{\Z^{n+1}}\leq_{2n+3} P_{\Z^n\cdot \Q}$. Thus, $P_{\Z^n\cdot \Q}$ has no $\Sigma^{in}_{2n+3}$ sentence and so it has Scott complexity $\Pi_{2n+3}$. 

    Similarly, the linear order $\omega^n$ has Scott complexity $\Pi_{2n+1}$ but $\omega^n+\omega^n\leq_{2n+1} \omega^n$ (see e.g.\ \cite{AK-book} Lemma 15.10). Thus, the same reasoning as above gives that $\SC(P_{\omega^n})=\Pi_{2n+2}$. 
\end{proof}

\begin{proposition}
\label{sigma_n sc}
Given a natural number $n\geq 5$, there exists a Presburger group with Scott complexity $\Sigma_n$. 
\end{proposition}

\begin{proof}
As shown in \cite{david-dino}, for every natural number $m\geq 4$ there is an infinite linear order $\LL$ such that $\SC(\LL)=\Sigma_m$ as well as an $\LL'\ncong \LL$ with $\LL\leq_m\LL'$. It follows from Lemma \ref{psr PL<1+ psr L} that $P_\LL$ has a $\Sigma_{m+1}$ Scott sentence and from Corollary \ref{Ls iff 1+ Ps} that $P_\LL\leq_{m+1} P_{\LL'}$.  Thus, $P_\LL$ has no $\Pi_{m+1}$ Scott sentence and so it has Scott complexity $\Sigma_{m+1}$. 
\end{proof}

\begin{proposition}
\label{d-sigma_n sc}
    Given a natural number $n\geq 2$, there exists a Presburger group with Scott complexity $d$-$\Sigma_n$. 
\end{proposition}

\begin{proof}
    We saw in Proposition \ref{P_1 SC} that $\SC(P_1)=d$-$\Sigma_2$. 
    Given $n\geq 2$, it was shown in the proof of Corollary 3.10 of \cite{david-dino} that there exists an infinite linear order $\LL$ such that $\SC(\LL)=d$-$\Sigma_n$ and linear orders $\LL',\LL''\ncong \LL$ with $\LL'\leq_n\LL\leq_n\LL''$. Thus, as in the above two proofs, we get that $P_\LL$ has a $\Pi^{in}_{n+2}$ Scott sentence and a $\Sigma^{in}_{n+2}$ Scott sentence, but no $\Pi^{in}_{n+1}$ nor $\Sigma^{in}_{n+1}$ Scott sentence. Hence, $\SC(P_\LL)=d$-$\Sigma_{n+1}$.  
\end{proof}

We now turn our attention to achieving the infinite Scott sentence complexities. We begin with the complexities whose indices are at least $2$ greater than a limit, for which we can make use of the following lemma: 

 \begin{lemma}[See e.g.\ \cite{CST} Table 1]
    \label{alpha+2}
    Let $\A$ be a countable structure and let $\alpha\geq 1$ be a countable ordinal.
    \begin{itemize}
        \item $\SC(\A)=\Sigma_{\alpha+2}$ if and only if $\pSR(\A)=\alpha$ and $\SR(\A)=\alpha + 2$. 
        \item $\SC(\A)=\Pi_{\alpha+2}$ if and only if $\pSR(\A)=\SR(\A)= \alpha+1$.
        \item $\SC(\A)=d$-$\Sigma_{\alpha+2}$ if and only if $\pSR(\A)=\alpha +1$ and $\SR(\A)=\alpha +2$. \qed
    \end{itemize}
        
    \end{lemma}

\begin{proposition}
\label{sc alpha+2}
    Given a countable ordinal $\alpha\geq \omega$, there exist Presburger groups with Scott sentence complexities $\Sigma_{\alpha+2}, \Pi_{\alpha+2}$ and $d$-$\Sigma_{\alpha+2} $.
\end{proposition}

\begin{proof}
    Fix any countable $\alpha\geq\omega$. It was shown in \cite{david-dino} that there exist countable linear orders $\LL_1,\LL_2$ and $\LL_3$ with Scott sentence complexities $\Sigma_{\alpha+2}, \Pi_{\alpha+2}$ and $d$-$\Sigma_{\alpha+2} $ respectively. By Lemma \ref{alpha+2}, we have that $\pSR(\LL_1)=\alpha$ and $\SR(\LL_1)=\alpha+2$ and so $\pSR(P_{\LL_1})=\alpha$  and $\SR(P_{\LL_1})=\alpha +2$ by Corollary \ref{psr=psr}. Hence, by Lemma \ref{alpha+2}, we have that $\SC(P_{\LL_1})=\Sigma_{\alpha+2}$. An analogous argument yields that $\SC(P_{\LL_2})=\Pi_{\alpha+2}$ and $\SC(P_{\LL_3})=d$-$\Sigma_{\alpha+2}$.
\end{proof}

As shown in \cite{Alvir-et}, no countable structure can have Scott sentence complexity $\Sigma_\lambda$ nor $d$-$\Sigma_\lambda$ for a limit ordinal $\lambda$. We now show the existence of Presburger groups with complexities  $d$-$\Sigma_{\lambda+1}$, $\Pi_{\lambda+1}$, $\Sigma_{\lambda+1}$ and $\Pi_{\lambda}$ for every countable limit ordinal $\lambda$. 

\begin{proposition}
\label{d sigma l + 1}
    Given a limit ordinal $\lambda$, there exists a Presburger group with Scott complexity $d$-$\Sigma_{\lambda+1}$. 
\end{proposition}

\begin{proof}
    We show that $P_{\omega^\lambda+\omega^\lambda}$ has the desired Scott complexity. Note, $\SR(\omega^\lambda+\omega^\lambda)=\lambda+1$ and $\pSR(\omega^\lambda + \omega^\lambda)=\lambda$ (see e.g.\ \cite{CST}), and so $\SR(P_{\omega^\lambda+\omega^\lambda})=\lambda+1$ and $\pSR(P_{\omega^\lambda+\omega^\lambda})= \lambda$ by Corollary \ref{psr=psr}. Thus, $P_{\omega^\lambda+\omega^\lambda}$ has both a $\Sigma_{\lambda+2}^{in}$ and a $\Pi_{\lambda+2}^{in}$ Scott sentence. However, we have $$\omega^\lambda +\omega^\lambda+ \omega^\lambda\leq_{\lambda+1} \omega^\lambda + \omega^\lambda\leq_{\lambda+1} \omega^\lambda $$ (see e.g.\ \cite{AK-book}) and so $$P_{\omega^\lambda+\omega^\lambda+\omega^\lambda}\leq_{\lambda+1} P_{\omega^\lambda+ \omega^\lambda}\leq_{\lambda +1} P_{\omega^\lambda}$$ by Corollary \ref{Ls iff 1+ Ps} (note $1+\lambda +1 = \lambda +1$). Hence, $P_{\omega^\lambda+\omega^\lambda}$ has no $\Sigma_{\lambda+1}^{in}$ nor $\Pi_{\lambda+1}^{in}$ Scott sentence and so $\SC(P_{\omega^\lambda +\omega^\lambda})=d$-$\Sigma_{\lambda +1}$. 
\end{proof}

\begin{proposition}
\label{pi l + 1}
    Given a limit ordinal $\lambda$, there exists a Presburger group with Scott complexity $\Pi_{\lambda+1}$. 
\end{proposition}

\begin{proof}
    We show that $P_{\omega^\lambda}$ has the desired Scott complexity. Since $\SR(\omega^\lambda)=\lambda$, $\SR(P_{\omega^\lambda})=\lambda$ by Corollary \ref{psr=psr} and so $P_{\omega^\lambda}$ has a $\Pi_{\lambda+1}^{in}$ Scott sentence. However since $P_{\omega^\lambda+\omega^\lambda}\leq_{\lambda+1} P_{\omega^\lambda}$, we have that $P_{\omega^\lambda}$ has no $\Sigma^{in}_{\lambda+1}$ Scott sentence and so $\SC(P_{\omega^\lambda})=\Pi_{\lambda+1}$. 
\end{proof}

The remaining cases, $\Sigma_{\lambda +1 }$ and $\Pi_\lambda$, are the most difficult. To achieve them, we find a method for converting a $\Sigma_{\lambda +1}^{in}$ or $\Pi^{in}_\lambda$ Scott sentence for $\LL$ into a Scott sentence for $P_\LL$ of the same complexity.

\begin{lemma}
\label{pull back}
    Given $\alpha\geq 1$, a linear order $\LL$ and a $\Pi_\alpha^{in} $ sentence $\phi$ in the language of linear orders, there exists a $\Pi_{1+\alpha}^{in}$ sentence $\phi^*$ in the language of Presburger arithmetic such that $\LL\models \phi$ if and only if $P_\LL\models \phi^*$. 
\end{lemma}

\begin{proof}
    This follows from the Pull-back theorem for jumps, as discussed in \cite{Maher2009}. We already have that there is a $\Delta^0_2$ embedding from the class of countable linear orders to the class of Presburger groups. Thus, the result in \cite{Maher2009} yields that any computable $\Pi_{\alpha}^{in}$ sentence $\phi$ can be turned into the needed $\phi^*$, and this $\phi^*$ will be computable $\Pi_{1+\alpha}^{in}$. However if $\phi$ is $\Pi_{\alpha}^{in}$ but not necessarily computably $\Pi_{\alpha}^{in}$, then the methods outlined in \cite{Maher2009} will still yield a (not necessarily computable) $\Pi_{1+\alpha}^{in}$ sentence $\phi^*$ such that $\LL\models \phi$ if and only if $P_\LL\models \phi^*$. 
\end{proof}

\begin{proposition}
\label{limit scott sentences transfer}
    Take any $\alpha\geq \omega$. If $\LL$ has a $\Pi^{in}_\alpha$ Scott sentence, then $P_\LL$ also has a $\Pi^{in}_\alpha$ Scott sentence. If $\LL$ has a a $\Sigma^{in}_\alpha$ Scott sentence, then $P_\LL$ also has a $\Sigma^{in}_\alpha$ Scott sentence.
\end{proposition}

\begin{proof}
    We prove the $\Pi^{in}_\alpha$ case. The proof of the $\Sigma^{in}_\alpha$ case is analogous. Suppose $\LL$ has a $\Pi^{in}_\alpha$ Scott sentence. Define $Pr$ to be the (infinite) conjunction of all the Presburger axioms. Note, $Pr$ is $\Pi^{in}_3$. Recall that $\rho(x)$ denotes the residue sequence of the element $x$ of a Presburger group. Define $$Plain:= (\forall x) \VV_{z\in \Z} \rho(x)=\rho(z) . $$ Given a Presburger group $P$, we have $P\models Plain$ if and only if $P$ is a plain Presburger group. Note, $Plain$ is also $\Pi^{in}_3$. 
    
    Recall that we write $x\sim y $ when $x$ and $y$ are archimedean equivalent and write $|x|$ to denote the max of $\{x,-x\}$. Define $$\Psi:= (\forall x,y)\left[\rho(x)\neq \rho(0)~ \vee~ \rho(y)\neq \rho(0) ~\vee ~ \neg(x\sim y) ~\vee ~ \VV_{n,m\in \N} n|x|=m|y|\right]. $$ Given a plain Presburger group $P$, we have that $P=V\times \Z$ for some divisible ordered abelian group $V$ by  Proposition \ref{plain}. If $P\models \Psi$ then given two elements of the form $x=(v,0)$ and $y=(u,0)$, if $x\sim y$ then  $y$ is a rational multiple of $x$. This holds if and only if $V$ has the form $V_{\mathcal{K}}$ for some linear order $\mathcal{K}$, and so $P\models \Psi$ if and only if $P\cong P_\mathcal{K}$ for some $\mathcal{K}$. Note, $\Psi$ is $\Pi^{in}_2$.
    
    Last, let $\phi$ be a $\Pi^{in}_\alpha$ Scott sentence of $\LL$. Let $\phi^*$ be as in the previous lemma. Since $1+\alpha=\alpha$, we have that $\phi^*$ is $\Pi^{in}_\alpha$. Define $$\Phi:= Pr ~\wedge ~ Plain~\wedge~\Psi ~\wedge~ \phi^*.$$ If $P\models \Phi$, then since $\Phi$ contains $Pr$, $Plain$ and $\Psi$, $P$ must have the form $P_{\mathcal{K}}$ for some linear order $\mathcal{K}$. Since $P\models \phi^*$, we have by the previous lemma that $\mathcal{K}\cong \LL$. Hence, $\Phi$ is a $\Pi_{\alpha}^{in}$ Scott sentence for $P_\LL$. 
\end{proof}

\begin{theorem}
\label{SCL < SCPL}
    Let $\LL$ be a countably infinite linear order. If $\SC(\LL)= \Gamma_{\alpha}$, then $\SC(P_\LL)\leq \Gamma_{1+\alpha}$ (for $\Gamma = \Pi, \Sigma$ or $d$-$\Sigma$).
\end{theorem}

\begin{proof}
    If $\alpha$ is infinite, then the $\Sigma_\alpha$ and $\Pi_\alpha$ cases follow from Proposition \ref{limit scott sentences transfer}. If $\alpha$ is finite, then the $\Sigma_\alpha$ case follows from the fact that $\pSR(P_\LL) \leq 1 +\pSR(\LL)$ (Lemma \ref{psr PL<1+ psr L}). The $\Pi_{\alpha}$ case follows from the fact that $\SR(P_\LL)\leq 1 +\SR(\LL)$ (Lemma \ref{less than 1 +}).

    For the $d$-$\Sigma_\alpha$ cases, recall that a structure has a $d$-$\Sigma_\alpha^{in}$ Scott sentence if and only if it has both a $\Pi^{in}_{\alpha+1}$ and a $\Sigma^{in}_{\alpha+1}$ Scott sentence. 
\end{proof}

\begin{proposition}
\label{Sigma Lambda + 1}
Given a limit ordinal $\lambda$, there exists a Presburger group with Scott complexity $\Sigma_{\lambda + 1}$.

\end{proposition}

\begin{proof}
    In \cite{tubo-david-mathew}, it is shown that there is a linear order $\LL$ with Scott complexity $\Sigma_{\lambda+1}$. By Theorem \ref{SCL < SCPL}, the Presburger group $P_\LL$ has a $\Sigma^{in}_{\lambda+1}$ Scott sentence. Since $\SC(\LL)=\Sigma_{\lambda+1}$, $\LL$ has no $\Pi^{in}_{\lambda+1}$ Scott sentence. Thus, $\SR(\LL)=\lambda +1$ and so $\SR(P_\LL)=\lambda +1$ by Lemma \ref{less than 1 +}. Hence, $P_\LL$ also has no $\Pi^{in}_{\lambda+1}$ Scott sentence and so $\SC(P_\LL)=\Sigma_{\lambda+1}$.
\end{proof}

\begin{proposition}
\label{Pi l}
    Given a limit ordinal $\lambda$, there exists a Presburger group with Scott complexity $\Pi_\lambda$.
\end{proposition}

\begin{proof}
    As shown in \cite{david-dino}, there exists a linear order $\LL$ with Scott complexity $\Pi_{\lambda}$. By Theorem \ref{SCL < SCPL}, the Presburger group $P_\LL$ has a $\Pi_{\lambda}^{in}$ Scott sentence. If $\SC(P_\LL)$ is not $\Pi_\lambda$, then since a Scott complexity cannot be $\Sigma_\lambda$ or $d$-$\Sigma_\lambda$, there would have to be some $\alpha<\lambda$ such that $P_\LL$ has a $\Pi^{in}_\alpha$ Scott sentence. This would imply that $\SR(P_\LL)\leq \alpha$ which implies that $\SR(\LL)\leq\alpha$ which contradicts that $\SC(\LL)=\Pi_\lambda$. Hence, $\SC(P_\LL)=\Pi_\lambda$. 
\end{proof}

\section{Non-Plain Presburger Groups}
\label{non-plain section}

With the exception of those constructed in Proposition \ref{up cone}, all Presburger groups discussed so far have been plain. That is, they only contained elements whose residue sequence was that of an integer. Since an element of a Presburger group can have any element of the continuum size group $\widehat{\Z}$ as its residue sequence, we have barely scratched the surface of the class of countable Presburger groups.   

Of particular interest are the recursively saturated Presburger groups. Given a Presburger group $P$ let $P_+$ denote the set of nonnegative elements of $P$. Note that $P_+$ is a model of $\Th(\N,+,<,0,1)$. It was shown by Lipshitz and Nadel in \cite{Lipshitz1978TheAS} that a countable nonstandard $P_+$ can be expanded to a model of Peano arithmetic if and only if it is recursively saturated.  Given any computable $\hat r \in \widehat\Z$, the type $T_{\hat r}(x)=\{(\E y)[ny + r_n=x]: n\in \N_{>0}\}$ is computable and finitely realizable in every Presburger group. Thus if $P$ is a recursively saturated Presburger group and $\hat r$ is computable, then $P$ contains an element with residue sequence $\hat r$. As shown by the following proposition, it can be readily seen that recursively saturated Presburger groups all have Scott rank 2.

\begin{proposition}
\label{recursively saturated}
    If $P$ is a countable recursively saturated Presburger group then $\SR(P)=2$. 
\end{proposition}

\begin{proof}
    Since $P$ is nonstandard, $\SR(P)\geq 2$ by Corollary \ref{non standard SR>1}. If $P$ is recursively saturated, then $P$ is homogeneous (see e.g.\ \cite{llew}). That is given $\vec p,\vec q \in P\lom$ of the same length, $\tp(\vec p)=\tp(\vec q)$ if and only if $\vec p$ is automorphic to $\vec q$. Additionally, since Presburger arithmetic has elimination of quantifiers down to $\Delta^0_1$ formulas (again see e.g.\ \cite{llew}), we get that $\vec p$ is automorphic to $\vec q$ if and only if they satisfy the same $\Delta^0_1$ formulas. Let $\Psi$ be the set of all $\Pi^0_1$ formulas satisfied by $\vec p$. We get that the automorphism orbit of $\vec p$ is $$\{ \vec x: P\models \WW_{\psi\in\Psi} \psi(\vec x)\}. $$ Thus, all automorphism orbits in $P$ are $\Sigma^{in}_2$-definable and so $\SR(P)\leq 2$ by Theorem \ref{robuster scott rank}.
\end{proof}

We now turn our attention to showing that there is no Presburger group with Scott complexity $\Sigma_3$. 
To do so, we must consider all possible Presburger groups, not just the plain ones. In \cite{Reed-et-al}, Goncharov, Lempp and Solomon consider computable ordered abelian groups in general. In particular, they show that a computable ordered abelian group is computably categorical if and only if it has a finite \emph{basis} (defined below). 

\begin{definition}
    Let $G$ be an ordered abelian group. Given $g_1,...,g_n\in G$, we say that $\{g_1,...,g_n\}$ is linearly independent if  given $k_1,...,k_n\in \Z$, then $$k_1g_1+\cdots + k_ng_n=0$$ if and only if each $k_i=0$.  A maximal linearly independent subset of $G$ is called a basis of $G$. The cardinality of a basis of $G$ is called the dimension of $G$. 
\end{definition}

Note, in \cite{Reed-et-al} they use the term \emph{rank} instead of dimension. However to avoid confusion with Scott rank, we will call it dimension. 
In \cite{Reed-et-al}, the authors show how given a computable $G$ of infinite dimension, one can build a computable $H\cong G$ such that there is no computable isomorphism from $G$ to $H$. The same constructions can be used to show how given any Turing degree $\ddd$ and $\ddd$-computable $G$, one can build a $\ddd$-computable $H$ such that there is no $\ddd$-computable isomorphism from $G$ to $H$. This yields the following: 

\begin{theorem}[Essentially Theorem 1.8 in \cite{Reed-et-al}]
If $G$ is a countable ordered abelian group, then $G$ is computably categorical on a cone if and only if $G$ has a finite basis. \qed
\end{theorem}

By Theorem \ref{robust pSR}, we get the following corollary:

\begin{corollary}
\label{cor: sigma_3 iff finite basis}
    An ordered abelian group has a $\Sigma^{in}_3$ Scott sentence if and only if it has a finite basis. \qed
\end{corollary}

Note that the standard Presburger group $\Z$ has dimension $1$, and all other Presburger groups have dimension at least $2$. In fact, the Presburger groups with dimension $2$ are exactly those of the form $\Z[\hat r]$ (as in Definition \ref{A[r]}). To see this let $P$ be a Presburger group with dimension $2$. There exists an infinite element $X\in P$ such that $\{1,X\}$ is a basis for $P$. It will follow that $P\cong \Z[\hat r]$ where $\rho(X)=\hat{r}$.

It should be noted that unlike with vector spaces, it is not necessary for every element of an ordered abelian group $G$ to be a linear combination of elements of a basis of $G$. For example, $\{2\}$ is a basis of $\Z$. However, we do get a similar property: 

\begin{lemma}[\cite{Reed-et-al}]
\label{dependence equation}
    Let $G$ be an ordered abelian group with basis $\{b_1,...,b_n\}$. Given any $g\in G$ with $g\neq 0_G$, there exists a unique integer sequence $(m,k_1,...,k_n)$ such that:
    \begin{itemize}
        \item $m>0$; and
        \item $\text{GCD}(m,k_1,...,k_n)=1$; and
        \item $mg=k_1 b_1+ \cdots +k_nb_n$.
    \end{itemize}
\end{lemma}
We call such an equation a \emph{(reduced) dependence equation} for $g$ over $\{b_1,...,b_n\}$.

\begin{definition}
    Given $(g_1,...,g_n)\in G^n$, we define $\cut(g_1,...,g_n)$ to be the set of $k_1,...,k_n\in \Z^n$ such that $$k_1g_1+\cdots + k_ng_n >0. $$
\end{definition}

\begin{theorem}
\label{thm: no sigma3}
    There is no Presburger group with Scott sentence complexity $\Sigma_3$.
\end{theorem}

\begin{proof}
    We need only show that any Presburger group with a $\Sigma^{in}_3$ Scott sentence also has a $\Pi^{in}_3$ Scott sentence. Fix such a $P$. We have by Corollary \ref{cor: sigma_3 iff finite basis} that $P$ has a finite basis. Fix elements $p_1,...,p_n$ which, along with $1$, form a basis of $P$. Note, this gives that $P$ has dimension $n+1$. Let the formula $\Phi(x_1,...,x_n)$ be the conjunction of the following for each $i\leq n$:

    \begin{itemize}
        \item $\rho(x_i)=\rho(p_i)$; and
        \item $\cut(1,x_1,...,x_n)=\cut(1,p_1,...,p_n)$; and
        \item The set $\{1,x_1,...,x_n\}$ is linearly independent.
    \end{itemize}
    Recall that $\rho(x)$ denotes the residue sequence of an element $x$ of a Presburger group. Note that $\Phi$ is $\Pi^{in}_1$. Let $T\subseteq\N$ be an oracle that codes each $\rho(p_i)$ as well as $\cut(1,p_1,...,p_n)$. Given a copy $\B\cong P$ and a tuple $\vec b \in \B^n$, it is decidable from $(\D(\B)\oplus T)'$ whether or not $\B\models \Phi(\vec b)$.
    
    Take $\A\cong P$ with $a_1,...,a_n\in \A$ such that $\A\models \Phi(a_1,...,a_n)$. Define $f:P\to \A$ such that $f(0_P)=0_\A$, $f(1_P)=1_\A$ and $f(p_i)=a_i$ for all $i$. Given any other $p\in P$, take the (unique) integers $m,k_0,k_1,...,k_n$ such that $mp=k_0+k_1p_1+\cdots +k_np_n$ with $m>0$ and GCD$(m,k_0,...,k_n)=1$ as in Lemma \ref{dependence equation}.  Define $f(p)$ to be the element $a\in \A$ such that $ma=k_0 + k_1a_1+ \cdots+ k_na_n$. The existence of $a$ follows from the fact that $\rho(a_i)=\rho(p_i)$ for each $i\leq n$ (which holds since $\A\models \Phi(a_1,...,a_n)$), and so $k_0 +k_1a_1 +\cdots + k_na_n\equiv k_0 + k_1p_1 +\cdots +k_np_n \equiv 0 \mod m$. The uniqueness of $a$ follows from the fact that these groups are torsion free.

    Thus we have that $f$ is well defined and injective. To see that $f$ is surjective, suppose there was some $b\in \A$ that was not in the image of $f$. We would have that $b$ does not satisfy any dependence equation over $1, a_1,...,a_n$. This would imply that the set $\{1,a_1,...,a_n,b\}$ is linearly independent which would contradict that $P$ (and thus $\A$) has dimension $n+1$. Hence, there is no such $b$ and we have that $f$ is a bijection. It is clear that addition is preserved under $f$. Furthermore if $q\in P$ with $q>0$, then $f(q)>0$ as well since $\cut(1,p_1,...,_n)=\cut(1,a_1,...,a_n)$. Hence, order is also preserved and so $f$ is an isomorphism. 

    Given two copies $\A,\B\cong P$, we can compute an isomorphism from $\A$ to $\B$ by simply finding tuples $\vec a\in \A^n$ and $\vec b\in \B^n$ with $\A\models \Phi(\vec a)$ and $\B\models \Phi(\vec b)$ and then following the same steps we used to build $f$ above. Note, such an isomorphism is (uniformly) computable from $(\D(\A)\oplus \D(\B) \oplus T)'$ which gives that $P$ is uniformly $\boldsymbol{\Delta}^0_2$-categorical. Hence, $P$ has a $\Pi^{in}_3$ Scott sentence by Theorem \ref{robuster scott rank}.   \end{proof}

\section{Some Degree Spectra of Presburger Groups}
\label{dgsp chapter}

\subsection{Computing Divisible Ordered Abelian Groups}

We can make use of the fact that $V\times \Z$ is a Presburger group when $V$ is a divisible ordered abelian group to build Presburger groups with some notable degree spectra. In particular, given any $S\subseteq \N$ we will be able to construct groups with degree spectra of the form $$\{\ddd: S \text{ is computable in } \ddd^{(\alpha)}\}$$ for all computable non-limit ordinals $\alpha\neq 2$ and $$\{\ddd: S \text{ is c.e.\ in } \ddd^{(\alpha)}\}$$ for all computable non-limit ordinals $\alpha\neq 0,2$ (Theorem \ref{dgsp theorem}). For the limit case, we will get a similar result (Corollary \ref{limit cor}).  Furthermore, given any countable linear order $\LL$, we will see by Theorem \ref{L'} that the Presburger group $P_\LL$ has degree spectrum $$\{\ddd:\ddd'\in \dgsp(\LL)\}.$$

As shown by the following proposition, being able to compute $V$ allows you to compute $V\times \Z$:

\begin{proposition}
\label{dgsp}
Consider the Presburger group $P=V \times \Z$ with $V$ a divisible ordered abelian group. We have $\dgsp(V)\subseteq \dgsp(P) $.
\end{proposition}

\begin{proof}
Suppose $\A\cong V$ and $\deg(\A)=\ddd$. It is clear that there is a $\ddd$-computable copy of the Presburger group $\A \times \Z\cong P$. Hence, $\dgsp(V)\subseteq \dgsp(P)$.
\end{proof}

\begin{proposition}
\label{prop d' in dgV}

    If $P$ and $V$ are as in the above proposition and $V$ is nontrivial, then $\{\ddd':\ddd\in \dgsp(P)\}\subseteq \dgsp(V)$.  
\end{proposition} 

\begin{proof}
    Let $\B\cong P$ with $\deg(\B)=\ddd$. Consider the substructure $\CC$ of $\B$ generated by the element $1_\B$. We have that $\CC\cong (\Z,+,<,0,1)$ and that $\CC$ is c.e.\ in $\B$. Hence, $\CC$ is $\ddd'$-computable and so the structure $\B/\CC$, which is isomorphic to $V$, is $\ddd'$-computable.
\end{proof}

\noindent This may lead one to wonder if $\dgsp(V)$ is the same as $\dgsp(P)$ with $P$ and $V$ as above. This question remains open for now (Question \ref{Q dgsp V vs dgsp P}).  

By a theorem of Holder \cite{holder}, each Archimedean ordered group is isomorphic to a subgroup of the additive ordered group on $\R$. Given such an ordered group, we  define its \emph{divisible closure} as follows:

\begin{definition}
\label{closure}
Let $R\subseteq \R$. We define the divisible closure $C(R)$ to be the smallest divisible ordered abelian group such that $R\subseteq C(R) \subseteq \R$. 
\end{definition}

\noindent For example, $C(\{1,\sqrt2 \})=\{a+b\sqrt{2}:a,b\in \Q\}$.
We previously saw by Proposition \ref{up cone} that every upper cone of Turing degrees is the degree spectrum of some Presburger group, though the groups constructed were not plain Presburger groups.  The following proposition uses an alternative construction that will yield plain Presburger groups with these spectra:

\begin{proposition}
\label{upper cone}
Given any Turing degree $\ddd$, there exists a divisible ordered abelian group whose degree spectrum is the cone above $\ddd$. 
\end{proposition}

\begin{proof}
Take $r_{\ddd} \in \R$ such that the set $\{q\in \Q: q<r_{\ddd}\}$ is of degree $\ddd$. For example, you can take $$r_{\ddd} =\sum_{i\in S} \frac{1}{10^i} $$ for some $S\subseteq \N$ of degree $\ddd$. Consider the divisible ordered abelian group $V_{\ddd} \cong C(\{1,r_{\ddd}\})$. Note that given an $\omega$-presentation $\A$ of a divisible ordered abelian group and $n,m\in \N$, the set $\{q\in \Q: qa_n<a_m\}$ is computable in $\deg(\A)$. Hence, given any copy $\B\cong V_{\ddd}
$, the atomic diagram $\D(\B)$ computes $\{q\in \Q: q<r_{\ddd}\}$. Hence, $\deg(\B)$ is at least $\ddd$ and so we have that the degree spectrum of $V_{\ddd}$  is a subset of the the cone above $\ddd$.

Since the degree $\ddd$ can compute (the Dedekind cut of) the real number $r_{\ddd}$, we can use $\ddd$ to construct a copy of $V_{\ddd}$. Thus, because the degree spectrum of a non-automorphically trivial structure is upwards closed, we have that the cone above $\ddd$ is a subset of the degree spectrum of $V_{\ddd}$. 
\end{proof}

\begin{corollary}
\label{upper cone cor}
Given any Turing degree $\ddd$, there exists a (plain) Presburger group whose degree spectrum is the cone above $\ddd$. 
\end{corollary}

\begin{proof}
Consider $V_{\ddd}\times \Z$. By Proposition \ref{dgsp}, we have that the cone above $\ddd$ is a subset of $\dgsp(V_{\ddd}\times \Z)$. Now suppose we have $\A\cong V_{\ddd}\times \Z$ with $\deg(\A)=\ccc.$ Note that there exists an isomorphism $f:V_{\ddd}\times \Z\to \A$ and $a_n,a_m\in \A$ such that $f(1,0)=a_n$ and $f(r_{\ddd},0)=a_m$. We have that $\D(\A)$ computes the set $\{q\in \Q: qa_n<a_m\}=\{q\in \Q:q<r_{\ddd}\}$. Hence, we have that $\ccc\geq \ddd$ and so all elements of $\dgsp(V_{\ddd}\times \Z)$ are in the cone above $\ddd$. 
\end{proof}

We can also use divisible ordered abelian groups to get Presburger groups with more interesting degree spectra. The following lemma will be useful in doing so:

\begin{lemma}
\label{c.e. in}
Let $\A$ be a copy of a divisible ordered abelian group.  If the atomic diagram $\D(\A)$ is c.e.\ in some degree $\ddd$, then $\D(\A)$ is $\ddd$-computable.    
\end{lemma}

\begin{proof}
We show how using a $\ddd$-oracle, one can determine whether or not $\A$ satisfies a given atomic sentence. Viewing $\A$ as purely relational, all elements of $\D(\A)$ are of one of the following forms:
\begin{align*}
   &a=0\\
   &a<b\\
   &a+b=c
\end{align*}
for some $a,b,c\in \A$.

Use the oracle to start enumerating the elements of $\D(\A)$.
Given $a\in \A$, to determine if $a=0$ is in $\D(\A)$, simply wait for a sentence of the form $a'=0$ to be enumerated. If $a'=a$, then $a=0$ is in $\D(\A)$, otherwise it is not. To see if $a+b=c$ is in $\D(A)$, wait until something of the form $a+b=c'$ to be enumerated. Check if $c=c'$. Since $<$ is a total order, to determine if $a<b$ is in $\D(\A)$, simply wait for either $a<b$ or $b<a$ to be enumerated. 
\end{proof}

\subsection{Degree Spectrum of $P_\LL$} Let $\LL$ be a  linear order. Recall that we define the divisible ordered abelian group $V_\LL$ as $\bigoplus_{l\in\LL}\Q$ and the Presburger group $P_\LL$ as $V_\LL\times \Z$ (see Definitions \ref{V_L} and \ref{P_l def}). We now show that $\dgsp(V_\LL)=\dgsp(P_\LL)=\{\ddd:\ddd'\in \dgsp(\LL)\}$. 

\begin{proposition}
\label{V_L  d' in L}
Given any countable linear order $\LL$,
    $\dgsp(V_\LL)=\{\ddd:\ddd'\in \dgsp(\LL)\}. $
\end{proposition}

\begin{proof}
    Let $\CC$ be a copy of $V_\LL$ of degree $\ddd$. The Archimedean equivalence relation $\sim$ (as in Definition \ref{arch equiv}) is c.e.\ in $\ddd$, and so $(\CC-\{0_\CC\})/\!\sim$, which is isomorphic to $\LL$ as a linear order, is $\ddd'$-computable. Hence, $\dgsp(V_\LL)\subseteq \{\ddd:\ddd'\in \dgsp(\LL)\}$.  

    For the other direction, suppose that $\ddd'\in \dgsp(\LL)$. It follows that there is some $\ddd$-computable set $T$ such that $T'=\{e:\Phi^T_e(e)\downarrow\}$ computes a copy $\A$ of $\LL$. Thus, there exists some Turing functional $\Psi$ such that $\Psi^{T'}(i,j)=1$ if $a_i<_\A a_j$ and $\Psi^{T'}(i,j)=0$ if $a_j<_\A a_i$. We use our oracle $T$ and functional $\Psi$ to compute the atomic diagram of a copy $\B$ of $V_\LL$. We do so by continually "guessing" at initial segments $\sigma$ of $T'$. We use the notation $\preceq$ and $\prec$ to denote initial segments and proper initial segments respectively.  Note, if we define $\sigma_{n,m}\in 2^n$ such that $\sigma_{n,m}(e)=1$ if and only if $\Phi^T_{e,m}(e)\downarrow$, we have that $\sigma_{n,m}$ will be an initial segment of $T'$ for sufficiently large $m$. 
    
    At each stage $s$ of the construction, we will have a set $H_s$ of elements in $\B$ of the form $h_{\sigma,i}$ for $i\in \N$ and $\sigma$ that still "look like" they may be initial segments of $T'$. That is, $\sigma$ such that $\sigma(e)=0$ if and only if $\Phi^T_{e,s}(e)\uparrow$. We treat all elements of $H_s$ as if they are in their own Archimedean equivalence class.  If at a later stage $t$ we have that $\Phi^T_{e,t}(e)\downarrow$ for an $e$ such that $\sigma(e)=0$, then we conclude that $\sigma$ is not an initial segment of $T'$ and we make it so each element of $H_t$ of the form $h_{\sigma, i}$ will actually be Archimedean equivalent to some other element of $H_t$ and we exclude $h_{\sigma,i}$ from $H_{t+1}$. Defining $H=\lim_sH_s$, we will have that all elements of $H$ are of the form $h_{\sigma,i}$ for $\sigma\prec T'$. Given $h_{\sigma,i}, h_{\rho,j}\in H$, we will have $h_{\sigma,i}\ll h_{\rho,j}$ if and only if $a_i<_\A a_j$. Additionally, each nonzero Archimedean equivalence class of $\B$ will contain exactly one element of $H$, and every element of $\B$ will be a finite $\Q$-linear combination of elements in $H$. Thus, $\B\cong \bigoplus_{i\in \A} \Q \cong V_\LL$. 

    For simplicity, we will view  $\B $ as purely relational. Furthermore, since $\B$ is divisible, the element $qx$ with $q\in \Q$ and $x\in \B_n$ is well defined and can be found computably in the atomic diagram $\D(\B)$.  
Thus, we may add multiplication by each $q\in \Q$ to our language. We get that the sentences in $\D(\B)$ will be those of the following forms:
\begin{itemize}
    \item $b_i=0$
    \item $b_i<b_j$
    \item $qb_i=b_j$
    \item $b_i + b_j = b_k$
\end{itemize}
where $b_i,b_j,b_k\in \B$ and $q\in \Q$.
We now give instructions for how, using our $\ddd$-oracle $T$, one can enumerate $\D(\B)$. Note, it will thus follow from Lemma \ref{c.e. in} that $\D(\B)$ is $\ddd$-computable.  \\

\paragraph{Construction:}  We let $I_s$ denote the set of possible initial segments of $T'$ through stage $s$. That is, $I_s=\{\sigma \in 2\lom: |\sigma|\leq s \text{ and } \Phi^T_{e,s}(e)\downarrow \text{ if and only if }\sigma(e)=1\}$. We make use of a function $f$ that will tell us for how large an $n\in \N$ the function $\Psi^\sigma$ can tell the order of all elements $a_i\in \A$ with $i<n$ when only running for $s$ steps. Specifically given $\sigma\in 2\lom$, we define $f(\sigma, s)$ to be the largest natural number $n\leq s$ such that $\Psi^\sigma(i,j)$ halts within $s$ steps for all $i,j<n$. Note, as $\sigma$ approaches $T'$ and $s$ approaches infinity, $f(\sigma,s)$ approaches infinity. 

For each $s>0$, define the set $X_s\subseteq I_t \times \N$ to be all elements of the form $(\sigma,m)$ that satisfy the following:

\begin{itemize}
    \item[(i)] $m\leq f(\sigma,s)$; and
    \item[(ii)] $m>f(\rho,s)$ for all $\rho\prec \sigma$.
\end{itemize}

\noindent Define $H_s=\{h_{\sigma,m}: (\sigma,m)\in X_s\}$. We will define an ordering on $H_s$ so that if $\rho\preceq \sigma$, then $h_{\rho,m}<h_{\sigma,n}$ if and only if $\Psi^\sigma (m,n)=1$. The purpose of condition (ii) is so $H=\lim H_s$ will contain exactly one element of the form $h_{\sigma, m}$ for each $m\in \N$ and so $H$ will be isomorphic to $\A$ as a linear order.

Let $\tau$ be the shortest initial segment of $T'$ such that $\Psi^\tau(0,1) \downarrow$. Let $r$ be the number of stages it takes for $\Psi^\tau(0,1)$ to halt. Define $t=\max(|\tau|,r)$. For reasons that will become clear later, we begin our construction at stage $t$ instead of stage $0$. \\

\subparagraph{Stage $t$:} Add the sentence $b_0=0$ to $\D(\B)$. For each $h \in H_t$, pick a fresh element $b_i \in \B$ and define $b_i=h$. Add the sentence $h>0$ to $\D(\B)$ for all $h\in H_t$. Given $h_{\sigma, m}, h_{\rho,n}\in H_t$, add the sentence $h_{\rho, m}< h_{\sigma,n}$ to $\D(\B)$ if either: 

\begin{itemize}
    \item $\rho \preceq \sigma $ and $\Psi^\sigma(m,n)=1$; or
    \item $\sigma \preceq \rho$ and $\Psi^\rho(m,n)=1$; or
    \item $\sigma$ and $\rho$ are incomparable but $\rho$ comes before $\sigma$ lexicographically.\\ 
\end{itemize}

\subparagraph{Stage $s+1>t$:}\,\\
\subparagraph{(i)} If $\sigma\in I_s - I_{s+1}$, then we will have that every $h_{\sigma,i}\in H_s$ will not be in $H_{s+1}$. Thus we arrange that all such $h_{\sigma,i}$ will be Archimedean equivalent to some $h\in H_{s+1}$. To do so, simply add the sentence $q h_{\sigma,i}= h$ for some sufficiently large (or small) $q\in \Q$ and appropriate $h\in H_{s+1}$ in such a way that respects the order on $H_s$. Note, our choice to begin the construction at stage $t$ guarantees that there is still at least one element in $H_{s+1}$. For every $r\in \Q$ such that the sentence $r h = a_c$ is in $\D(\B)$, add the sentences $rqh_{\sigma, i}=b_c$ and $\frac{1}{rq}b_c=h_{\sigma,i}$ to $\D(\B)$ as well. 

Given $h_{\sigma, m}, h_{\rho,n}\in H_{s+1}$ such that neither $h_{\sigma, m}< h_{\rho, n}$ nor $h_{\rho, n}<h_{\sigma, m}$ is in $\D(\B)$, add the sentence $h_{\rho, m}< h_{\sigma,n}$ to $\D(\B)$ if either: 

\begin{itemize}
    \item $\rho \preceq \sigma $ and $\Psi^\sigma(m,n)=1$; or
    \item $\sigma \preceq \rho$ and $\Psi^\rho(m,n)=1$; or
    \item $\sigma$ and $\rho$ are incomparable but $\rho$ comes before $\sigma$ lexicographically.\\ 
\end{itemize}

\subparagraph{(ii)} Let $(q_i)_{i\in \N}$ be an enumeration of the rationals. For each $h\in H_{s+1}$, take the least $i$ such that there is no sentence of the form $q_i h=b$ in $\D(\B)$. Take a fresh element $b_c$ and add the sentences $q_i h = b_c$ and $\frac{1}{q_i} b_c=h$ to $\D(\B)$. \\

\subparagraph{(iii)} For every finite tuple $(h_1,...,h_n)\in H\lom$ with each $h_i$ unique and every tuple $(q_1,...,q_n)\in\Q^n $ such that a sentence of the form $q_i h_i= b$ is in $\D(\B)$ for each $i$ but there is no sentence of the form $q_1h_1 + \cdots +q_nh_n= b $ in $\D(\B)$, take a fresh element $b_c$ and add the sentence $q_1h_1 + \cdots +q_nh_n= b_c$ to $\D(\B)$. \\

\subparagraph{(iv)} If we have $b,b'\in \B$ that have been previously mentioned in the construction but we have neither $b<b'$ nor $b'<b$ in $\D(\B)$, we must now decide the ordering on $b,b'$. Fix an enumeration $h_1,...,h_n$ of all the elements of $H$. We have that there exists $q_1,...,q_n,q'_1,...q'_n\in \Q$ such that $b=q_nh_n+ \cdots +q_1 h_1$ and $b'= q'_nh_n + \cdots +q'_1 h_1$. Add the sentence $b<b'$ to $\D(\B)$ if and only if $(q_n,...,q_1)$ is lexicographically less than $(q'_n,...,q'_1)$. 

Last, add $b_i+ 0 = b_i$, $1b_i=b_i$ and $0b_i= 0$ to $\D(\B)$ if they are not already in. If $b_i+b_j=b_c$ is in $\D(\B)$, add $b_j + b_i = b_c$ as well if it is not already in. If $qb_i=b_j$ is in $\D(\B)$ with $q\neq 0$ add $\frac{1}{q}b_j=b_i$ as well if it is not already in. This completes the construction.\\

\paragraph{Verification}
Define $H=\lim_sH_s$. Note that $h_{\sigma,n}\in H$ only if $\sigma\prec T'$. Furthermore, if $h_{\rho,m}\in H$ with $\rho\preceq\sigma$, then $h_{\rho,m}<_\B h_{\sigma,n}$ if and only if $\Psi^\tau(m,n)=1$ if and only if $\Psi^{T'}(m,n)=1$ if and only if $a_m<_\A a_n$. Hence, $(H,<_\B)$ is isomorphic to $\LL$. Additionally, our instructions give that $H$ is exactly the Archimedean rank of $\B$, and every element of $\B$ has the form $q_1 h_1 + \cdots + q_nh_n$ with each $q_i\in \Q$ and $h_i\in H$. Hence, $\B\cong \bigoplus_{h\in H} \Q \cong \bigoplus_{l\in \LL} \Q = V_\LL$. 
\end{proof}

\begin{theorem}
\label{L'}
    Given any countable linear order $\LL$,
    $\dgsp(P_\LL)=\dgsp(V_\LL)=\{\ddd:\ddd'\in \dgsp(\LL)\}. $
\end{theorem}

\begin{proof}
    Since $P_\LL= V_\LL \times \Z$, we have by Proposition \ref{dgsp} that $\dgsp(V_\LL)\subseteq\dgsp(P_\LL)$. Now suppose that $\ddd\in \dgsp(P_\LL)$ and that $\A$ is a $\ddd$-computable copy of $P_\LL$. Note that $\AR(P_\LL)= \LL^*$ where $\LL^*=\LL\cup \{X\}$ where $X$ is a left-endpoint not in $\LL$. Since Archimedean equivalence in $\A$ is c.e.\ in $\D(\A)$, we have that $\LL^*$ has a $\ddd'$-computable copy, and so it is clear that $\LL$ must have a $\ddd'$-computable copy as well.  
\end{proof}

This gives that computing a copy of $P_\LL$ is often easier than computing $\LL$. For example in \cite{Russell-D2}, R. Miller constructed a linear order $\mathcal{K}$ that has no computable copy but does have copies of every noncomputable $\Delta^0_2$-degree. Since $\dgsp(\mathcal{K})$ contains $\boldsymbol{0}'$, we get that $\dgsp(P_\mathcal{K})$ is simply the collection of all Turing degrees.

Theorem \ref{L'} allows us to leverage existing results about the degree spectra of linear orders to build degree spectra of Presburger groups. For example Ash, Downey, Jockusch and Knight showed in  
\cite{jumps-of-orderings}, \cite{AK} and \cite{Jump-degrees} that given a computable ordinal $\alpha\geq 2$ and a degree $\ddd\geq \boldsymbol{0}^{(\alpha)}$, there is a linear ordering having $\alpha^{th}$ jump degree $\ddd$ sharply. In particular: 

\begin{proposition}[Lemma 2.1 in \cite{Jump-degrees}]
\label{Jump-degrees lemma}
    Given any computable ordinal $\beta\geq 2$ that is not a limit ordinal and any set $S\subseteq \N$, there exists a linear order $\LL(\beta,S)$ such that $\dgsp(\LL(\beta, S))=\{\ddd: S \text{ is c.e.\ in } \ddd^{(\beta)}\}.$ \qed
\end{proposition}

\begin{corollary}
\label{jumps cor}
Given any computable ordinal $\alpha\geq 3$ that is not a limit ordinal and any set $S\subseteq \N$, there exists: \begin{itemize}
    \item[(1)] A divisible ordered abelian group $V$ and a Presburger group $P=V\times \Z$ such that $\dgsp(V)=\dgsp(P)=\{\ddd:S \text{ is c.e.\ in } \ddd^{(\alpha)}\} $; and
    \item[(2)] A divisible ordered abelian group $V$ and a Presburger group $P=V\times \Z$ such that $\dgsp(V)=\dgsp(P)=\{\ddd:S \text{ is computable in } \ddd^{(\alpha)}\} $.
\end{itemize}
\end{corollary}

\begin{proof}
    Take $\beta$ such that $1+\beta= \alpha$. For (1), take $V=V_{\LL(\beta,S)}$. We have by Proposition \ref{L'} that $\dgsp(V)=\dgsp(P)=\{\ddd: \ddd'\in \dgsp(\LL(\beta,S))\}=\{\ddd: S \text{ is c.e.\ in }(\ddd')^{(\beta)}\} = \{\ddd:S \text{ is c.e.\ in } \ddd^{(\alpha)}\} $.  For (2), take $V=V_{\LL(\beta, S \oplus \overline{S})}$ and note that $S\oplus \overline{S}$ is c.e.\ in $\ddd^{(\alpha)}$ if and only if $S$ is computable in $\ddd^{(\alpha)}$. 
\end{proof}

We can get a similar result for the case where $\alpha$ is a limit ordinal. Given a limit ordinal $\lambda$, we say that an increasing sequence $(\alpha_n)_{n\in \N}$ is a \emph{fundamental sequence} for $\lambda$ if each $\alpha_n<\lambda$ and $\bigcup_n \alpha_n=\lambda$. 

\begin{proposition}[Lemma 2.4 in \cite{Jump-degrees}] Given any limit ordinal $\lambda$ and any set $S\subseteq \N\times \N$, there exists a fundamental sequence $(\alpha_n)_{n\in \N}$ converging to $\lambda$ and a linear order $\LL(\lambda,S)$ such that $$\dgsp(\LL(\lambda,S))=\{\ddd: S_n \text{ is computable in $\ddd^{(\alpha_n)}$ uniformly in $n$ for every $n\in \N$}\},$$ where $S_n=\{m:(n,m)\in S\}$. \qed

\end{proposition}

\begin{corollary}
\label{limit cor}
    Given any limit ordinal $\lambda$ and any set $S\subseteq \N\times \N$, there exists a fundamental sequence $(\alpha_n)_{n\in \N}$ converging to $\lambda$, a divisible ordered abelian group $V$, and a Presburger group $P=V\times \Z$ such that $\dgsp(V)=\dgsp(P)=\{\ddd: S_n \text{ is computable in $\ddd^{(\alpha_n)}$ uniformly in $n$ for every $n\in \N$}\}$, where $S_n=\{m:(n,m)\in S\}$. \qed
\end{corollary}

\subsection{One-Jump Spectra} 

Absent from Corollary \ref{jumps cor} are the cases where $\alpha$ is $0, 1$ or $2$. This is largely due to the absence of the $\alpha=0,1$ cases from Proposition \ref{Jump-degrees lemma}. It was shown by Richter in \cite{Richter_1981} and by Knight in \cite{Knight_1986} that the sets $\{\ddd: S \text{ is computable in } \ddd\}$ and $\{\ddd: S \text{ is computable in } \ddd'\}$ are the degree spectra of linear orders if and only if $S$ is computable. Thus, we must look beyond structures of the form $V_\LL$ and $P_\LL$ to fill in these gaps in Corollary \ref{jumps cor}. 

We already saw by Propositions \ref{up cone} and \ref{upper cone} and Corollary \ref{upper cone cor} that for any $S\subseteq \N$, the set $\{\ddd: S \text{ is computable in } \ddd\}$ is the degree spectrum of a divisible ordered abelian group as well as a Presburger group. We now turn our attention to the case involving a single jump of $\ddd$. 

\begin{proposition}
\label{degrees}
Let $S\subseteq \N$. The set of Turing degrees $\{\ddd: S$ is c.e.\ in $\ddd'\}$ is the degree spectrum of a divisible ordered abelian group.

\end{proposition}

\begin{proof}
    Let $\{p_n\}_{n\in \N}$ be an enumeration of the primes. Recall that given $R\subseteq\R$, we define $C(R)$ to be the divisible closure of $R$ (Definition \ref{closure}). Define $$V_S=\bigoplus_{n\in \N} U_n$$ where $$U_n = \begin{cases} C(\{1,\sqrt{p_n}\}) & \text{if } n\in S\\
\Q & \text{if } n\notin S.
\end{cases}$$ Suppose $\A$ is a copy of $V_S$ of degree $\ddd$. For each $n\in \N$, define the formula $$\phi_n := (\E x,y)\Big[\WW_{\sqrt{p_n}>q\in \Q} q x <y \,\,\&\,\, \WW_{\sqrt{p_n}<q\in \Q} q x> y   \Big].  $$ 
We have that $\A\models \phi_n$ if and only if $n\in S$. Note that each $\phi_n$ is $\Sigma^c_2$. That is, each $\phi_n$ is a \emph{computable infinitary} $\Sigma^{in}_2$ sentence (for background on computable infinitary logic, see e.g.\ \cite{AK-book}).  Hence, $\ddd'$ can enumerate the set of $n$ such that $\A\models \phi_n$ and can therefore enumerate $S$. 

For the other direction,  suppose that $S$ is c.e.\ in $\ddd'$. It follows that there exists a $\ddd$-computable set $T$ such that $\Fin^T = \{e: \dom(\Phi^T_e) \text{ is finite}\}= S$, as $\Fin$ is a known $\Sigma^0_2$-complete set. We will now build a $\ddd$-computable copy $\B\cong V_S$. We do so by building $\ddd$-computable copies $\B_n\cong U_n$ for all $n$ and define $\B=\bigoplus_{n\in \N} \B_n$. Since each $\B_n$ will be $\ddd$-computable, $\B$ will be $\ddd$-computable as well.

 We begin the construction of $\B_n$ by adding in an element $X$ that we treat like $1\in U_n$. At each stage $s+1$ of the construction, we have an element $h_s$  that we have been treating like $\sqrt{p_n}$. If the domain of $\Phi_n^T$ gains no new elements at stage $s+1$, then we let $h_{s+1}=h_s$ and keep treating this element like $\sqrt{p_n}$. If the domain of $\Phi_n^T$ does gain a new element at stage $s+1$, then we stop treating $h_s$ like it is $\sqrt{p_n}$ by finding a ``suitable'' $q\in \Q$  and adding the sentence $qX=h_s$. By a suitable $q$, we just mean one such that adding the sentence $qX=h_s$ will not contradict any of the sentences we have already added. Such a $q$ can be found by taking some rational that is sufficiently close to $1/\sqrt{p_n}$. We then find a fresh element to call $h_{s+1}$ which we begin to treat like $\sqrt{p_n}$. 

Note, if the domain of $\Phi_n^T$ is finite, then there will be a sufficiently large $k$ such that $h_{k}=h_{m}$ for all $m\geq k$. This element will turn out to satisfy ``$\sqrt{p_n}X=h_{k}$.'' In the case where $\Phi_n^T$ is infinite, every $h_s$ will turn out to be a rational multiple of $X$, and thus no element will be ``$\sqrt{p_n}X$.''

For simplicity, we will view each $\B_n $ as purely relational. Furthermore, since each $\B_n$ is divisible, the element $qx$ with $q\in \Q$ and $x\in \B_n$ is well defined and can be found computably in the atomic diagram $\D(\B_n)$.  
Thus, we may add multiplication by each $q\in \Q$ to our language. We get that the  sentences in $\D(\B_n)$ will be those of the following forms:
\begin{itemize}
    \item $b_i=0$
    \item $b_i<b_j$
    \item $qb_i=b_j$
    \item $b_i + b_j = b_k$
\end{itemize}
where $b_i,b_j,b_k\in \B_n$ and $q\in \Q$.
We now give instructions for how, using a $\ddd$ oracle, one can enumerate $\D(\B_n)$. Note, it will thus follow from Lemma \ref{c.e. in} that $\D(\B_n)$ is $\ddd$-computable.  \\

\paragraph{Construction:}  Fix an oracle $T\in \ddd$. Let $W_{n,m}^T$ denote the domain of $\Phi_n^T$ after running for $m$ steps. Let $\{u_{i}\}_{i\in \N}$ be a computable decreasing, and  $\{l_{i}\}_{i\in \N}$  a computable increasing, sequence of rationals such that  $\lim_{i}u_{i}=\lim_{i} l_{i}=\sqrt{p_n}$. Throughout the construction, we will have an increasing sequence $\{k_i\}_{i\in \N}$ that will keep the interval $[l_{k_i},u_{k_i}]$ sufficiently tight. \\

\subparagraph{Stage $0$:} Add the sentences $b_0= 0_{\B_n}$ and $b_0<b_1$ to $\D(\B_n)$. We will henceforth refer to $b_1$ as $X$. Let $k_0=0$ and let $b_2=h_{0}$.\\

\subparagraph{Stage $s+1$:}~ \\

\subparagraph{(i)} Check if $|W^T_{n,s+1}|>|W^T_{n,s}|$. 
\begin{itemize}
    \item If no, let $k_{s+1}=k_s + 1$ and let $h_{s+1}=h_s$.
    \item If yes, find the smallest natural number $k_{s+1}> k_s$ such that given any $q,q'\in \Q$ with the sentences $qX=b_i$ and $q'h_s=b_j$ already in $\D(\B_n)$ for some $b_i,b_j$, we have $q\neq q'u_{k_{s+1}}$. Add $u_{k_{s+1}} X = h_s$ to $\D(\B_n)$. Additionally, for all $q\in \Q$ such that there exists $b_j$ with $qh_s=b_j$ in $\D(\B_n)$, add the sentence $u_{k_{s+1}} q X = b_j$ to $\D(\B_n)$.  Find a fresh element of $\B_n$ to call $h_{s+1}$. \\
\end{itemize}

\subparagraph{(ii)} Let $\{q_i\}_{i\in \N}$ be an enumeration of $\Q$. Let $C=\{q\in \Q: qX=b_i \text{ is in $\D(\B_n)$ for some $b_i$} \}  $ and $D=\{q\in \Q: qh_{s+1}=b_i \text{ is in $\D(\B_n)$ for some $b_i$} \}  $. Find the least $m,p\in \N$ such that:
\begin{itemize}
    \item $-\frac{r}{t}$ is not in the interval $[l_{{k_{s+1}}}, u_{k_{s+1}}]$ for any $r\in D\cup \{q_m\}, t\in C\cup \{q_p\}$ (this will help us decide which elements are positive later on); and
    \item The sentences $q_mX=b_i$ and $q_p h_{s+1}= b_j$ are not yet in $\D(\B_n)$ for any $b_i,b_j$.
\end{itemize}
 Using fresh elements $b_c$ and $b_d$, add the sentences $q_mX=b_c$ and $q_p h_{s+1}=b_d$ to $\D(\B_n)$.\\

\subparagraph{(iii)} Suppose we have the sentences $b_i=qX$ and $b_j=q'h_{s+1}$ in $\D(\B_n)$ but no element equal to $qX + q'h_{s+1}$. Take a fresh element $b_c$ and add the sentence $b_c=b_i + b_j
$ to $\D(\B_n)$.\\

\subparagraph{(iv)} Suppose $b\in \B_n$ has appeared in a sentence already in $\D(\B_n)$. We have that $b$ can be expressed as $b=qX + q'h_{s+1}$ with $q,q'\in \Q$. Suppose we have such elements $b_i$,$b_j$, and $b_c$ with $$b_i=q_iX + q'_ih_{s+1}, $$
 $$b_j=q_jX + q'_jh_{s+1}, $$
and
$$b_c=q_kX + q'_jh_{s+1}. $$
If $q_i + q_j = q_c$ and $q'_i+q'_j=q'_c$, add the sentence $b_i+b_j=b_c$ to $\D(\B_n)$. If there is some $r\in \Q$ with $rq_i =q_j$ and $rq'_i=q'_j$, add the sentence $rb_i=b_j$ to $\D(\B_n)$. 

Note, by our careful selection of $q_m$ and $q_p$ during part (ii) of each stage, we are guaranteed that $-\frac{q_i}{q'_i}$ is not in the interval $[l_{{k_{s+1}}},u_{k_{s+1}}]$.  We say that $b_i>0$ if 
\begin{itemize}
    \item $q'_i=0$ and $q_i>0$; or
    \item $-\frac{q_i}{q'_i}<l_{k_{s+1}}$. 
\end{itemize}
If $b_i + b_j= b_c$ and $b_j>0$, then add the sentence $b_i<b_c$ to $\D(\B_n)$. 

Last, add $b_i+ 0 = b_i$, $1b_i=b_i$ and $0b_i= 0$ to $\D(\B_n)$ if it is not already in. If $b_i+b_j=b_c$ is in $\D(\B_n)$, add $b_j + b_i = b_c$ as well if it is not already in. If $qb_i=b_j$ is in $\D(\B_n)$ with $q\neq 0$, add $\frac{1}{q}b_j=b_i$ as well if it is not already in. This completes the construction.\\

\paragraph{Verification:} If $W_n^T$ is finite, then define $h=\lim_s h_s$. If $W_n^T$ is infinite, then we say that $h$ does not exist. All elements $b\in \B_n$ can be expressed as $b=qX + q'h$ (or just as $b=qX$ if $W_n^T$ is infinite) for unique $q,q'\in \Q$. Suppose $i\neq j$ but $b_i$ and $b_j$ both equal $qX + q'h$ (or just $qX$ if $h$ does not exist). Now suppose that $h=h_s$ exists and that $q,q'\neq 0$. It is clear from our construction that $b_i,b_j$ must have been first mentioned as fresh elements to satisfy $b_{i'}+b_{i''}=b_i$ and $b_{j'} + b_{j''}=b_j$ with $b_{i'},b_{j'}=qX$ and $b_{i''},b_{j''}= q'h_s$ at part (iii) of some stages. It is clear from our instructions that only one element of $\B_n$ will equal $q'h_s$. Hence we get that $i''=j''$. Note, if $i'=j'$, then $b_i=b_j$ because only one element will ever be added satisfying $b_{i'}+ b_{i''}$. Hence, we must have that $i'\neq j'$ yet the sentences $qX=b_{i'}$ and $qX=b_{j'}$ both are in $\D(\B_n)$. Without loss of generality, suppose that $qX=b_{j'}$ enters $\D(\B_n)$ at a later stage than $qX=b_{i'}$. Call this stage $p$. Note, it is not possible for $qX=b_{j'}$ to have entered at part (ii) as there was already an element equal to $qX$. This means that $p<s$ and it must have entered at part (i) due to $b_{j'}=q''h_{k_p}$ and $q''u_{k_{p+1}}=q$. However, we choose $k_{p+1}$ at part (i) of stage $p+1$ to avoid this happening. Hence, this is not possible and we get that if $b_i,b_j=qX + q'h$, then $b_i=b_j$.

Furthermore, it is clear from our instructions that there will be elements equal to $qX$ and $qh$ (or just $qX$ if $h$ does not exist) for all $q\in \Q$. Because of part (iii) of the construction, there will also be elements equal to $qX + q'h$ if $h$ exists for all $q,q'\in \Q$. Thus, we get that the maps $qX + q'h \mapsto q + q'\sqrt{p_n} $ and $qX\mapsto q$ are bijections from $\B_n$ to $C(1,\sqrt{p_n})$ and $\Q$  in the cases where $h$ does and does not exist respectively. Furthermore, it is clear from our construction that these maps are isomorphisms of divisible ordered abelian groups. Hence, we get that $\B_n\cong C(1,\sqrt{p_n})$ if $W^T_n$ is finite and $\B_n\cong \Q$ otherwise. 
\end{proof} 

\begin{corollary}
\label{degree cor 1}
    Let $S\subseteq \N$. The set of Turing degrees $\{\ddd: S$ is c.e.\ in $\ddd'\}$ is the degree spectrum of a Presburger group.
\end{corollary}

\begin{proof}
    Consider the structure $P= V_S\times \Z$ where $V_S$ is as in the previous proof. By Proposition \ref{dgsp}, we have that $\dgsp(V_S)=\{\ddd: S$ is c.e.\ in $\ddd'\}\subseteq \dgsp(P)$. 

    For the other direction, suppose that $\A\cong P$. We will have that $n\in S$ if and only if $$\A\models (\E x,y)\left[\WW_{\sqrt{p_n}>\frac{a}{b}\in \Q} ax<by \,\,\& \WW_{\sqrt{p_n}<\frac{a}{b}\in \Q} ax>by \right].$$ In particular, if $x$ and $y$ are witnesses to this sentence in $V_S$ then $(x,0)$ and $(y,0)$ are witnesses in $P$. This sentence is $\Sigma^c_2$. Hence, we get that $S$ is c.e.\ in $\deg(\A)'$. 
\end{proof}

\begin{corollary}
\label{degree cor 2}
Let $S\subseteq \N$. The set of Turing degrees $\{\ddd: S$ is computable in $\ddd'\}$ is the degree spectrum of a divisible ordered abelian group. 
\end{corollary}

\begin{proof}
Consider  $V_{S\oplus \overline{S}}$ (as defined in the previous proof). We have that the degree spectrum of $V_{S\oplus \overline{S}}$ is $\{\ddd:S\oplus \overline{S}$ is c.e.\ in $\ddd'\}=\{\ddd: S$ is computable in $\ddd'$\}. 
\end{proof}

\begin{corollary}
\label{degrees cor 3}

Let $S\subseteq \N$.  The set of Turing degrees $\{\ddd: S$ is computable in $\ddd'\}$ is the degree spectrum of a Presburger group. \qed

\end{corollary}

Combining these with our previous results, we get the following:

\begin{theorem}
\label{dgsp theorem}
    Let $\alpha$ be a computable ordinal that is not a limit ordinal, and let $S$ be a subset of $\N$. \begin{itemize}
        \item[(1)] If $\alpha \neq 0,2$, then there exists a  divisible ordered abelian group $V$ and a Presburger group $P=V\times \Z$ such that $\dgsp(V)=\dgsp(P)=\{\ddd:S \text{ is c.e.\ in } \ddd^{(\alpha)}\}$.
        \item[(2)] If $\alpha \neq 2$, then there exists a divisible ordered abelian group $V$ and a Presburger group $P=V\times \Z$ such that $\dgsp(V)=\dgsp(P)=\{\ddd:S \text{ is computable in } \ddd^{(\alpha)}\} $.
    \end{itemize}
\end{theorem}

\begin{proof}
    For (1), the case of $\alpha=1$ follows from Proposition \ref{degrees} and Corollary \ref{degree cor 1}. The remaining cases are given in Corollary \ref{jumps cor}. For (2), the case of $\alpha=0$ follows from Proposition \ref{upper cone} and Corollary \ref{upper cone cor}, and the case of $\alpha=1$ follows from Corollaries \ref{degree cor 2} and \ref{degrees cor 3}. The remaining cases are given in Corollary \ref{jumps cor}. 
\end{proof}

Recall that we saw a similar result above in Corollary \ref{limit cor} for the case where $\alpha$ is a limit ordinal. 

\section{Further Questions}
\label{questions}

\subsection{Scott Analysis}

\begin{question}
\label{Q Sigma2/3}
    Are there any Presburger groups with Scott sentence complexity $\Sigma_4$?
\end{question}

Since there are no linear orders with Scott sentence complexity  $\Sigma_3$, our methods used in Section \ref{Scott Complexities Section} cannot yield us such a Presburger group of the form $P_\LL$. However, there are many other forms that Presburger groups can take. 

\begin{question} Let $\LL$ be a countable linear order.
\label{Q SRPL vs SRL}
\begin{itemize}
    \item[(i)] Must we have $\SR(P_\LL)=1+\SR(\LL)$?
    \item[(ii)] If $|\LL|>1$ and $\SC(\LL)=\Gamma_{\alpha}$ (with $\Gamma= \Pi, \Sigma$ or $d$-$\Sigma$), must we have $\SC(P_\LL)=\Gamma_{1+\alpha}$?
\end{itemize}
\end{question}

Note that if $\LL$ is the linear order containing only a single point, then $\SC(\LL)=\Pi_1$ but $\SC(P_\LL)=\SC(P_1)=d$-$\Sigma_2$ by Proposition \ref{P_1 SC}. All linear orders considered in this paper satisfy both (i) and (ii), and we have that (i) will hold so long as $\SR(\LL)$ is infinite (see Corollary \ref{psr=psr}). However, must this hold for all linear orders? 

Question \ref{Q SRPL vs SRL} is closely related to the following open question posed by Gonzalez and Harrison-Trainor:

\begin{question}[Question 1.10 in \cite{spectral-gaps}]
 \label{Q david/matthew}
 Let $\A$ be a countable structure with the property that whenever $\B$ is a countable structure and
$\A\leq_\alpha \B$ then $\B\cong \A$. Does $\A$ have a $\Pi_\alpha$ Scott sentence?
     
 \end{question}

If the answer to Question \ref{Q david/matthew} is yes, then given $\LL$ with $\SR(\LL)=n$ there would have to be some linear order $\K\ncong \LL$ such that $\LL\leq_{n} \K$. Thus, $P_\LL\leq_{1+n} P_\K$ by Corollary \ref{Ls iff 1+ Ps} and so $\SR(P_\LL)\geq 1+n$. This, along with the fact that $\SR(P_\LL)\leq 1+   \SR(\LL)$ (Lemma \ref{psr PL<1+ psr L}), would yield a positive answer to Question \ref{Q SRPL vs SRL} part (i).    

\subsection{Degree Spectra}

\begin{question}
\label{Q dgsp V vs dgsp P}
    Let $V$ be a nontrivial divisible ordered abelian group, and let $P$ be the Presburger group $V\times \Z$. Must we have $\dgsp(V)=\dgsp(P)$? 
\end{question}

All of the Presburger groups of the form $V\times \Z$ that we examined in this paper have this property. Additionally, we saw by Propositions \ref{dgsp} and \ref{prop d' in dgV} that $\dgsp(V)\subseteq\dgsp(P)$ and that $\{\ddd':\ddd\in \dgsp(P)\}\subseteq \dgsp(V)$. However, it
is not obvious how to strengthen Proposition \ref{prop d' in dgV}. Although $V$ is always co-c.e.\ in $P$, it is not clear how to computably enumerate $V$ given the atomic diagram of $P$.

\begin{question}
\label{Q universal dgsp}
    Are Presburger groups universal for degree spectra?
\end{question}

We call a class of structures $\frak{C}$ \emph{universal for degree spectra} if for every countable non-automorphically trivial structure $\B$, there exists some $\A\in \frak{C}$ such that $\dgsp(\A)\cong \dgsp(\B)$. It was shown by Hirschfeldt, Khoussainov, Shore and Slinko in \cite{hirschfeldt-et-al} that the following classes of structures are universal for degree spectra: symmetric irreflexive graphs, partial orderings, rings, integral domains, commutative semi-groups, and 2-step nilpotent groups. It was further shown by R. Miller, Poonen, Schoutens and Shlapentokh in \cite{Miller-fields} that the class of fields is universal for degree spectra as well. It should be noted that linear orders are not universal for degree spectra, so we must look beyond Presburger groups of the form $P_\LL$ to answer Question \ref{Q universal dgsp}.

\bibliographystyle{plain}
\bibliography{mybib}

\pdfoutput=1
\end{document}